\documentclass[12pt]{article}
\usepackage{mathtools,amsthm,amssymb}
\usepackage[margin=1.2in]{geometry}
\usepackage{xcolor,graphicx,tikz,tikz-cd}
 \usetikzlibrary{patterns,arrows,matrix}
 \usetikzlibrary{decorations.markings}

\usepackage[colorlinks, linkcolor=red, citecolor=brown]{hyperref}

\usepackage{mathrsfs,bbm}
\usepackage{enumerate}

\theoremstyle{plain}
 \newtheorem{theo+}{Theorem}
 \numberwithin{theo+}{section}
 \newtheorem{prop+}[theo+]{Proposition}
 \newtheorem{coro+}[theo+]{Corollary}
 \newtheorem{lemm+}[theo+]{Lemma}

\theoremstyle{definition}
 \newtheorem{defi+}[theo+]{Definition}
 
 \newtheorem{not+}[theo+]{Notation}
 \newtheorem{rema+}[theo+]{Remark}
 \newtheorem{example}[theo+]{Example}

\newenvironment{theorem}{\begin{theo+}}{\end{theo+}}
\newenvironment{proposition}{\begin{prop+}}{\end{prop+}}
\newenvironment{corollary}{\begin{coro+}}{\end{coro+}}
\newenvironment{lemma}{\begin{lemm+}}{\end{lemm+}}
\newenvironment{remark}{\begin{rema+}}{\end{rema+}}
\newenvironment{definition}{\begin{defi+}}{\end{defi+}}

\renewcommand{\mathcal}{\mathscr}
\newcommand{\Ae}{\mathcal{A}^e}
\newcommand{\Ac}{\mathcal{A}^\mathrm{c}}
\newcommand{\Ai}{\mathcal{A}^\mathrm{i}}
\newcommand{\BZ}{\mathcal{B}}
\renewcommand{\dh}{\widehat{d}}
\newcommand{\Eh}{\widehat{E}}
\newcommand{\kap}{\boldsymbol{\kappa}}

\newcommand {\fM}{\mathfrak{M}}
\newcommand {\fI}{\mathfrak{I}}
\newcommand{\Pol}{\mathcal{P}}
\newcommand{\ds}{\displaystyle}
\newcommand{\col}{:}
\newcommand {\K} {{\mathbf{k}}}

\newcommand{\Z}{\mathbb{Z}}
\newcommand{\Zp}{\Z_{\ge 0}}

\newcommand{\R}{\mathbb{R}}
\newcommand{\cA}{\mathcal{A}}

\newcommand{\cE}{\mathcal{E}}
\newcommand{\cEh}{\widehat{\mathcal{E}}}

\newcommand{\cM}{\mathcal{M}}
\newcommand{\cP}{\mathcal{P}}

\newcommand{\ba}{\mathbf{a}}
\newcommand{\bb}{\mathbf{b}}
\newcommand{\be}{\mathbf{e}}
\newcommand{\bff}{\mathbf{f}}

\DeclareMathOperator{\Ker}{Ker}
\DeclareMathOperator{\ord}{ord}

\numberwithin{equation}{section}

\title{\bf Bizonotopal Graphical Algebras}

\author{Anatol Kirillov%
  \thanks{Yanqi Lake Beijing Institute of Mathematical Sciences and Applications, Huairou District, Beijing, China, \ \texttt{\href{mailto:kirillov@bimsa.cn}{kirillov@bimsa.cn}}}
  \and
  Gleb Nenashev%
  \thanks{Department of Mathematics and Computer Science, St. Petersburg State University, St. Petersburg, 199178, Russia, \ \ \texttt{\href{mailto:glebnen@gmail.com}{glebnen@gmail.com}}}
  \and
  Boris Shapiro%
  \thanks{Department of Mathematics, Stockholm University, SE-106 91 Stockholm, Sweden \ and \ Department of Mathematics with Computer Science,
    Guangdong Technion-Israel Institute of Technology,     Shantou, China, \
    \  {\texttt{\href{mailto:shapiro@math.su.se} {shapiro@math.su.se}}}}
  \and
  Arkady Vaintrob%
  \thanks{Department of Mathematics, University of Oregon, Eugene, OR 97403, USA, \ \texttt{\href{mailto:vaintrob@uoregon.edu} {vaintrob@uoregon.edu}}}
}

\date{}

\begin{document}

\maketitle

\begin{abstract}
  Zonotopal algebras (external, central, and internal) of an undirected graph~$G$, introduced by Postnikov--Shapiro and Holtz--Ron, are finite-dimensional commutative graded algebras whose Hilbert series encode
  a wealth of combinatorial information about $G$.
  In this paper, we associate to $G$ a new family of algebras, which we call \emph{bizonotopal}, since their definition involves doubling the set of edges of~$G$.
  These algebras are monomial and exhibit intricate properties related, among other things, to the combinatorics of graphical parking functions and their associated polytopes.

  Unlike classical zonotopal algebras, the Hilbert series of bizonotopal algebras are not specializations of the Tutte polynomial of~$G$.
  Nevertheless, we show that in the external and central cases these Hilbert series satisfy a modified deletion--contraction relation.
  In addition, we prove that the external bizonotopal algebra is a complete graph invariant.
\end{abstract}

\section{Introduction}

Let $G=(V,E)$ be a finite undirected graph, possibly with loops and multiple edges, with vertex set $V$ and edge set $E$.
The \emph{external zonotopal algebra} $\Ae_G$ of $G$, introduced in~\cite{PSS}, is the commutative graded algebra defined as follows.

Let $\K$ be a field of characteristic zero. Consider the \emph{edge algebra}
\begin{equation}\label{eq:sqfree}
  \cE_G := \K[E]/(x^2_e \col e\in E),
\end{equation}
the quotient of the polynomial algebra $\K[E]:=\K[x_e \col e\in E]$ in the \emph{edge variables} $x_e$, $e\in E$, by the ideal generated by their squares.
This algebra has a basis consisting of the monomials $\ds \prod_{e\in S}x_e$, indexed by subsets $S\subset E$.
In particular, $\dim \cE_G = 2^{|E|}$ and its Hilbert series is $(1+t)^{|E|}$.

Fix an orientation of~$G$, and let $A_G=(a_{v,e})$ be the corresponding \emph{oriented incidence matrix} of~$G$, with rows indexed by vertices $v\in V$ and columns indexed by edges $e\in E$.
The entries of $A_G$ are given by
\begin{equation}\label{eq:incid}
  a_{v,e}=\begin{cases}
   -1, & \text{if the edge $e$ starts at $v$},\\
    1, & \text{if the edge $e$ ends at $v$},\\
    0, & \text{if $e$ is a loop or is not incident to $v$}.
\end{cases}
\end{equation}

The algebra $\Ae_G$ is defined as the subalgebra of $\cE_G$ generated by the elements
\begin{equation}\label{eq:gener}
  y_v := \sum_{e\in E} a_{v,e}\, x_e \in \cE_G, \qquad v\in V.
\end{equation}
Reversing the orientation of an edge $e\in E$ produces an isomorphic graded algebra, with the isomorphism induced by changing the sign of the corresponding generator $x_e\in\cE_G$.
Consequently, the isomorphism class of $\Ae_G$ as a graded algebra is independent of the choice of orientation.

It was shown in~\cite{PSS} that $\dim \Ae_G$ is equal to the number of spanning forests of~$G$, and that the Hilbert series of $\Ae_G$ is a specialization of the Tutte polynomial of~$G$ enumerating spanning forests by their \emph{external activity}.
In particular, this Hilbert series depends only on the graphical matroid of~$G$.
Furthermore, Nenashev~\cite{Ne} proved that the external zonotopal algebras of two graphs $G_1$ and $G_2$ are isomorphic if and only if the matroids of $G_1$ and $G_2$ are isomorphic.

The algebra $\Ae_G$ also admits a presentation as a quotient
\begin{equation}\label{eq:present}
  \Ae_G \simeq \K[V]/I_G^e,
\end{equation}
where $\K[V]:=\K[z_v \col v\in V]$ is the polynomial algebra in vertex-labeled variables $z_v$, and $I_G^e$ is the ideal generated by powers of linear forms
\begin{equation}\label{eq:power}
  I_G^e := \left( \Bigl( \sum_{v\in S} z_v \Bigr)^{d_S+1} \col \emptyset \neq S \subset V \right).
\end{equation}
Here $d_S$ denotes the number of edges of $G$ with one endpoint in $S$ and the other in $V\setminus S$.

The \emph{central} and \emph{internal} zonotopal algebras $\Ac_G$ and $\Ai_G$ of~$G$, introduced in~\cite{PS,HoRo}, admit analogous presentations as quotients of $\K[V]$ by power ideals $I_G^c$ and $I_G^i$, obtained from~\eqref{eq:power} by replacing the exponent $d_S+1$ with $d_S$ and $d_S-1$, respectively.
More precisely,
\begin{equation*}
  I_G^c := \left( \sum_{v\in V} z_v,\ \Bigl( \sum_{v\in S} z_v \Bigr)^{d_S} \col S \subsetneq V \right),
\end{equation*}
and
\begin{equation*}
  I_G^i := \left( \sum_{v\in V} z_v,\ \Bigl( \sum_{v\in S} z_v \Bigr)^{d_S-1} \col S \subsetneq V \right).
\end{equation*}

These graded algebras also encode essential combinatorial information about~$G$, and their Hilbert series are again specializations of the Tutte polynomial.
In particular, the dimension of the central zonotopal algebra $\Ac_G$ is equal to the number of spanning trees of~$G$.

In this paper, we introduce three new finite-dimensional commutative graded algebras associated with a graph~$G$: the \emph{external} $\BZ_G^e$, the \emph{central} $\BZ_G^c$, and the \emph{internal} $\BZ_G^i$ \emph{bizonotopal} algebras, so named because their definition involves a doubling of the edge set of~$G$.

These algebras are related to the usual zonotopal algebras of~$G$, but they encode substantially different information.

For example, the dimension of the highest-degree component of the external bizonotopal algebra $\BZ_G^e$ is equal to the number of spanning forests of~$G$ (that is, the total dimension of the algebra $\Ae_G$), while the dimension of the top-degree component of the central bizonotopal algebra $\BZ_G^c$ is equal to the number of spanning trees of~$G$ (which coincides with the total dimension of $\Ac_G$).
Both of these numerical invariants are determined by the Tutte polynomial of~$G$.
However, unlike the case of ordinary zonotopal algebras, the Hilbert series of the bizonotopal algebras are not specializations of the Tutte polynomial and contain information about~$G$ not captured by its graphical matroid.

For example, we show that the dimension of the external bizonotopal algebra $\BZ_G^e$ is equal to the number of lattice points in the convex hull of the set of \emph{weak parking functions} of $G$, a concept which we introduce and study in Section~\ref{sec:park}.
For the complete graph $K_n$, weak parking functions coincide with the usual parking functions, and hence $\dim \BZ_{K_n}^e$ is equal to the number of lattice points in the parking-function polytope studied in~\cite{AW}.

Moreover, we prove that the external bizonotopal algebra is a complete invariant of graphs without isolated vertices, and that the central bizonotopal algebra distinguishes graphs in which all vertices have degree at least two.

In contrast to ordinary zonotopal algebras, bizonotopal algebras are monomial, that is, they are isomorphic to quotients of polynomial algebras by monomial ideals.
Somewhat unexpectedly, the Hilbert series of the external and central bizonotopal algebras satisfy a form of a deletion--contraction relation, similar to but subtly different from the classical one satisfied by the Tutte polynomial.

\medskip

We conclude the introduction by briefly explaining the origin of our constructions.
They come from the work of the first author~\cite{Ki} on certain quadratic algebras related to integrable systems and Schubert calculus.
In~\cite[\S 4.3.3]{Ki}, it was observed that the external zonotopal algebra of the complete graph $K_n$ arises as a quotient of the subalgebra generated by additive Dunkl elements in the quadratic algebra $6T_n$, associated with unitary solutions of the classical Yang--Baxter equation.
It was further suggested that an analogous commutative subquotient of the corresponding non-unitary quadratic algebra might also be of interest.
The resulting algebra is precisely the external bizonotopal algebra of the complete graph.

\medskip

The paper is organized as follows.
In Section~\ref{sec:ex}, we introduce and study the external bizonotopal algebras $\BZ_G^e$.
We prove that $\BZ_G^e$ is a complete graph invariant and that it admits a monomial basis in bijection with the set of lattice points in the polytope of partial score vectors of $G$, which, as we also prove, coincides with the convex hull of weak $G$-parking functions.
In Section~\ref{sec:r-bizon}, we introduce the central and internal bizonotopal algebras as members of a larger family of $r$-bizonotopal algebras, with the external algebra corresponding to $r=1$.
We prove that for $r\ge 0$ the Hilbert series of these algebras satisfy a recursion, which we call the \emph{loopy deletion--contraction relation}.
In the external case, we also present an exact sequence that ``categorifies'' this relation.
In Section~\ref{sec:add}, we establish several additional results concerning the central and internal algebras.
In Section~\ref{sec:conclude}, we formulate a number of open questions for further study.
Finally, in the appendix (Section~\ref{sec:num}), we collect computational results for the Hilbert series of the external, central, and internal bizonotopal algebras of complete graphs with at most nine vertices.

\medskip\noindent
\emph{Acknowledgements.}
The second and third authors are grateful to the Beijing Institute for Mathematical Sciences and Applications (BIMSA) for hospitality during June--July 2024, when work on this project began.
The research of the third author was partially supported by the Swedish Research Council grant~2021-04900.

\section{External bizonotopal algebra}\label{sec:ex}

In this section, we define and begin to study the external bizonotopal algebra $\BZ_G^e$ of a graph $G$. We prove that this algebra is monomial and show that it has a basis corresponding to partial score vectors of $G$ or, equivalently, to lattice points in the convex hull of weak $G$-parking functions.

\subsection{Definitions and basic properties}
\subsubsection{Conventions and notation}

All algebras and vector spaces in this paper are defined over a fixed field $\K$ of characteristic zero. We denote by $\Zp$ the set of non-negative integers and by $[n]$ the set $\{1,2,\ldots,n\}$ of the first $n$ natural numbers.

By a graph $G=(V,E)$ in this paper we understand a finite undirected graph  with vertex set $V$ and edge set $E$, possibly with loops and multiple edges.
For a vertex $v\in V$, we denote by $\ell(v)$ the number of loops at $v$ and by $d(v)$ the number of edges connecting $v$ with vertices $u\ne v$.

We define the \emph{degree of a subset} $S\subset V$ of
vertices of $G$ to be the number $\kap_S$ of edges incident to at least one vertex of $S$, i.e.\
\begin{equation}
  \label{eq:degs}
  \kap_S:=\bigl|\{e\in E \col v\in e \text{\ for some\  } v\in S\}\bigr|.
\end{equation}
To simplify notation, for a singleton set $S=\{v\}$, its degree will be denoted simply as $\kap_v$.  Note that, since each loop is counted in $\kap_S$ only once, we have
$$\kap_v=d(v)+\ell(v),$$
which in general is different from the degree of vertex $v$ in the traditional sense.

\subsubsection{Definition of $\BZ_G^e$}
\label{sec:definition}

For a graph  $G=(V,E)$, denote by $\Eh$ the set of its oriented edges (which we also call arrows), i.e.\ edges with all possible orientations.
Let $$s: \Eh\to V$$ be the map sending an oriented edge $e\in \Eh$ to its source $s(e)\in V$ and  let
$$': \Eh\to \Eh$$ be the involution reversing the orientation of $e\in \Eh$.
Thus, the vertex $s(e')$ is the target of the oriented edge $e$.
 The orbits  of  the involution $'$ can be identified with the set $E$ of edges of $G$  and its fixed points correspond to the loops of $G$.
 Thus the number of oriented edges is equal to
 \begin{equation}
   \label{eq:oriented}
   |\Eh|=2|E|-\ell,
 \end{equation}
 where
 $$\ds \ell=\sum_{v\in V}\ell(v)$$
 is the total number of loops in $G$.

 In this notation, an orientation of a graph $G$ is simply a section
 $$ \omega: E\to \Eh $$
 of the projection map
\begin{equation}
\label{eq:proj}
  \pi:\Eh\to E, \ \pi(e):=\{e,e'\}\in E
\end{equation}
forgetting the direction of an arrow $e\in \Eh$. We will also need a slightly more general notion.
\begin{definition}
  A \emph{partial orientation} of a graph $G$ is a choice of orientations for a subset of its edges $S\subset E$, i.e., a section
  $  S\to \Eh $   of the restriction $\pi\vert_{\pi^{-1}(S)}$.
 Equivalently, it can be viewed as a subset of arrows $\Sigma \subset \Eh$ such that the restriction $\pi\vert_\Sigma$ of the  projection~\eqref{eq:proj} is injective.
\end{definition}

Similar to the square-free algebra~\eqref{eq:sqfree}, we consider the \emph{partial orientation algebra}
\begin{equation}
  \label{eq:partial}
  \cEh_G:=\K[\Eh]/( x_e^2, x_ex_{e'} \col  e\in \Eh ),
\end{equation}
the quotient of the polynomial algebra
$$\K[\Eh]:=\K[x_e\col e\in \Eh]$$
in the \emph{arrow variables} $x_e$ by the ideal generated by their squares and the products $x_ex_{e'}$  corresponding to different orientations of the same edge.
The following immediate proposition explains the naming of $\cEh_G$.

\begin{proposition}
  \label{prop:partial}
For each subset $\Sigma\subset \Eh$, consider the monomial
\begin{equation}
  \label{eq:monom}
  \ds x_\Sigma:=\prod_{e\in \Sigma} x_e \in \K[\Eh].
\end{equation}
(1)
The image of $x_\Sigma$ in $\cEh_G$ is nonzero if and only if $\Sigma$ is a partial orientation. Moreover, the images of the monomials $x_\Sigma$ corresponding to distinct partial orientations of $G$ form a basis of the algebra $\cEh_G$.
\\[4pt]
(2)  As a graded algebra, $\cEh_G$ is isomorphic to  the tensor product of $\ell$ copies of the algebra of dual numbers $D:=\K[\varepsilon]/(\varepsilon^2)$   and $|E|-\ell$ copies of the three-dimensional algebra
$T: = \K[\varepsilon,\varepsilon']/(\varepsilon^2, (\varepsilon')^2, \varepsilon\varepsilon')$,
i.e.\
  \begin{equation}
    \label{eq:tensor}
  \cEh_G\simeq (D^{\otimes \ell})\otimes (T^{\otimes (|E|-\ell)}).
  \end{equation}
  \\[4pt]
  (3)
  The dimension of $\cEh_G$ is equal to $2^\ell3^{|E|-\ell}$ and its Hilbert series is equal to $(1+t)^\ell(1+2t)^{|E|-\ell}$.
\end{proposition}
  \qed

\noindent
To each vertex $v\in V$ we associate a degree one element
\begin{equation}\label{eq:gens}
y_v=\sum_{e\in s^{-1}(v)}x_e~,
\end{equation}
in the algebra $\cEh_G$,
i.e.\ the sum of the generators $x_e$ over all arrows with source $v$.
(To avoid notational clutter, we use $x_e$ to denote generators of $\cEh_G$. The difference should be clear from the context.)

\begin{definition} [{\sf The algebra $\BZ^e_G$}]
\label{def:external}
 For  a graph  $G$,  the  subalgebra $\BZ^e_G$ of the partial orientation algebra $\cEh_G$ generated by the elements $y_v, \ v\in V$, is called the \emph{external bizonotopal algebra} of  $G$.
\end{definition}

Clearly, $\BZ^e_G$ is a finite-dimensional graded algebra. It has various connections with ordinary zonotopal algebras.
For example, any choice of an orientation $\omega: E\to \Eh$  of $G$ induces a homomorphism  from $\cEh_G$ onto the edge algebra $\cE_G$~\eqref{eq:sqfree} given by
$$f_\omega: \cEh_G\to \cE_G, \ f_\omega(x_e)
:=\begin{cases} x_{\pi(e)}, & \mathrm{\ if\ } \omega(s(e))=e\\
  -x_{\pi(e)}, &  \mathrm{\ if\ } \omega(s(e))\ne e~,\\
\end{cases}
$$
which induces a surjective homomorphism  $\BZ_G^e\twoheadrightarrow \Ae_G$
onto the usual external zonotopal algebra.

Similarly, the \emph{orientation forgetting homomorphism} induced by $\pi$
$$  \cEh_G\to \cE_G, \ x_e\mapsto  x_{\pi(e)},$$
gives a projection $\BZ_G^e\twoheadrightarrow \cA^+_G$ onto the algebra $\cA^+_G$ constructed similarly to $\Ae_G$ with the oriented incidence matrix~\eqref{eq:incid} replaced by the unoriented one (see~\cite{ShVa}). The algebra $\cA^+_G$ is the external zonotopal algebra of the \emph{even-circle matroid} of $G$~\cite{Do,SP}.

 \subsubsection{Basis of $\BZ_G^e$}
We now describe a basis of $\BZ_G$ and identify it with the set of partial score vectors of $G$.

\begin{definition}[{\sf Partial score vectors}]
Let $G=(V,E)$ be a graph with $n=|V|$ vertices.  An  $n$-tuple $(a_v)_{v\in V}\in \Zp^V$ is called a
\emph{partial score vector} of $G$ if there exists
a partial orientation $\Sigma\subset \Eh$ of $G$ such that $a_v$
is equal to the number of arrows in $\Sigma$ starting at $v$.
\end{definition}

Partial score vectors arise as exponents of basis elements of $\BZ_G^e$.
For a vector $\ba=(a_v)\in \Zp^V$, consider the monomial
$$\ds  y^\ba:=\prod_{v\in V} y_v^{a_v}\in \BZ_G^e$$
in generators \eqref{eq:gens}.

\begin{proposition}\label{lem:score}
\  \begin{enumerate}[(1)]
\item
The monomial $y^\ba\in \BZ_G^e$ is nonzero if and only if $\ba$ is a partial score vector.
\item The elements  $y^\ba$ of $\BZ_G^e$ corresponding to different partial score vectors of $G$ are linearly independent and hence form a basis of $\BZ_G^e$.
\end{enumerate}
\end{proposition}
\begin{proof}
  We start with a simple but crucial observation.
  If $e\in \Eh$ is an arrow with source $v=s(e)$ and target $u=s(e')$, then $y_v$ is the only generator~\eqref{eq:gens} that contains the variable $x_e\in \cEh_G$.
  Therefore, a nonzero monomial $x_\Sigma$~\eqref{eq:monom} in variables $x_e$ (with $\Sigma\subset \Eh$)
appears in the expansion of a monomial  $y^\ba$ exactly when $\ba$ is the partial score vector of the partial orientation $\Sigma$.
  This shows that  $y^\ba\ne 0$ exactly when $\ba$ is a partial score vector. Moreover, since the sets of monomials $x_\Sigma$ appearing in expansions of different nonzero elements $y^\ba$ are disjoint, Proposition~\ref{prop:partial} implies that these elements are linearly independent. This proves (1) and (2).
\end{proof}

\begin{corollary} \label{cor:hilb-e}
\
  \begin{enumerate}[(1)]
  \item
    The dimension of $\BZ_G^e$ is equal to the number of partial score vectors of $G$.
  \item
    The dimension of the $m$-th graded component of $\BZ_G^e$ is equal to the number of partial score vectors $\ba$ of weight $m$, where the \textbf{weight} of $\ba\in \Zp^V$ is
    \[
      |\ba|:=\sum_{v\in V} a_v.
    \]
  \item
    The top degree of $\BZ_G^e$ is equal to $|E|$.
  \item
    The dimension of the top degree component of $\BZ_G^e$ is equal to the number of spanning forests of $G$.
  \end{enumerate}
\end{corollary}

\begin{proof}
Part~(1) follows from Proposition~\ref{lem:score}(2).
The grading of $\BZ_G^e$ is induced from the polynomial algebra $\K[\Eh]$, which implies (2).

Part (3) follows from the fact that a partial orientation $\Sigma\subset \Eh$ can have at most $|E|$ oriented edges. So basis elements of maximal degree correspond to total orientations of $G$ and thus the dimension of the top degree component of $\BZ_G^e$ is equal to the number of usual score vectors of $G$. By the result of Kleitman and Winston~\cite{KlWi}, this number is equal to the number of spanning forests of $G$, thus giving (4).
\end{proof}

\

\begin{example}
For illustration, consider the external bizonotopal algebra of the graph
    \[G=
     \parbox[c]{1.3in}{  \begin{tikzpicture}[scale=.7]
        \coordinate (A) at (0, 0);
        \coordinate (B) at (2, 0);
        \coordinate (C) at (4, 0);
        \draw (A) -- (B);
        \draw (B) -- (C);
        \draw(C) to[out=45,in=135,loop, min distance=1.5cm] (C);
        \filldraw[black] (A) circle (2pt) node[below] {\small 1};
        \filldraw[black] (B) circle (2pt) node[below] {\small 2};
        \filldraw[black] (C) circle (2pt) node[below] {\small 3};
     \end{tikzpicture}}
.
\]

\noindent
The algebra $\BZ_G^e$ is the subalgebra of the partial orientation algebra
$$\cEh_G=\K[x_{12}, x_{21}, x_{23}, x_{32}, x_{33}]/(x_{ij}^2,x_{12}x_{21}, x_{23}x_{32})$$
generated by the elements
    $$
    y_1=x_{12},  \   y_2=x_{21}+x_{23},  \   y_3=x_{32}+x_{33}.
    $$
It has dimension $13$ and top degree $3$.
 Bases of graded components of  $\BZ_G^e$ with corresponding partial score vectors $\ba=(a_1,a_2,a_3)$ and examples of partial orientations producing them are given in the following table. The top degree component has dimension $4$ which is the same as the number of spanning forests in $G$.

\medskip

 \newcommand{\gzzz}{ \parbox[c]{.8in}{\begin{tikzpicture}[scale=.4]
       \coordinate (A) at (0, 0); \coordinate (B) at (2, 0);  \coordinate (C) at (4, 0);
       \draw (C) to[out=135,in=45,loop, min distance=1.5cm] (C);
       \draw (A)--(B)--(C);
       \filldraw[black] (A) circle (1.5pt) node[below] {\small 1};
       \filldraw[black] (B) circle (1.5pt) node[below] {\small 2};
       \filldraw[black] (C) circle (1.5pt) node[below] {\small 3};
     \end{tikzpicture}   }}
 \newcommand{\gozz}{ \parbox[c]{.8in}{\begin{tikzpicture}[scale=.4]
       \coordinate (A) at (0, 0); \coordinate (B) at (2, 0);  \coordinate (C) at (4, 0);
       \draw (C) to[out=135,in=45,loop, min distance=1.5cm] (C);
       \draw[-latex] (B)--(A);       \draw(C)--(B);
       \filldraw[black] (A) circle (1.5pt) node[below] {\small 1};
       \filldraw[black] (B) circle (1.5pt) node[below] {\small 2};
       \filldraw[black] (C) circle (1.5pt) node[below] {\small 3};
     \end{tikzpicture}   }}
 \newcommand{\gzoz}{ \parbox[c]{.8in}{\begin{tikzpicture}[scale=.4]
       \coordinate (A) at (0, 0); \coordinate (B) at (2, 0);  \coordinate (C) at (4, 0);
       \draw (C) to[out=135,in=45,loop, min distance=1.5cm] (C);
       \draw[-latex] (A)--(B);       \draw(C)--(B);
       \filldraw[black] (A) circle (1.5pt) node[below] {\small 1};
       \filldraw[black] (B) circle (1.5pt) node[below] {\small 2};
       \filldraw[black] (C) circle (1.5pt) node[below] {\small 3};
     \end{tikzpicture}   }}
 \newcommand{\gzzo}{ \parbox[c]{.8in}{\begin{tikzpicture}[scale=.4]
       \coordinate (A) at (0, 0); \coordinate (B) at (2, 0);  \coordinate (C) at (4, 0);
       \draw [-latex](C) to[out=135,in=45,loop, min distance=1.5cm] (C);
       \draw (A)--(B);       \draw (C)--(B);
       \filldraw[black] (A) circle (1.5pt) node[below] {\small 1};
       \filldraw[black] (B) circle (1.5pt) node[below] {\small 2};
       \filldraw[black] (C) circle (1.5pt) node[below] {\small 3};
     \end{tikzpicture}   }}
 \newcommand{\gooz}{ \parbox[c]{.8in}{\begin{tikzpicture}[scale=.4]
       \coordinate (A) at (0, 0); \coordinate (B) at (2, 0);  \coordinate (C) at (4, 0);
       \draw (C) to[out=135,in=45,loop, min distance=1.5cm] (C);
       \draw[-latex] (B)--(A);       \draw[-latex] (C)--(B);
       \filldraw[black] (A) circle (1.5pt) node[below] {\small 1};
       \filldraw[black] (B) circle (1.5pt) node[below] {\small 2};
       \filldraw[black] (C) circle (1.5pt) node[below] {\small 3};
     \end{tikzpicture}   }}
 \newcommand{\gozo}{ \parbox[c]{.8in}{\begin{tikzpicture}[scale=.4]
       \coordinate (A) at (0, 0); \coordinate (B) at (2, 0);  \coordinate (C) at (4, 0);
       \draw [-latex](C) to[out=135,in=45,loop, min distance=1.5cm] (C);
       \draw[-latex] (B)--(A);     \draw(C)--(B);
       \filldraw[black] (A) circle (1.5pt) node[below] {\small 1};
       \filldraw[black] (B) circle (1.5pt) node[below] {\small 2};
       \filldraw[black] (C) circle (1.5pt) node[below] {\small 3};
     \end{tikzpicture}   }}
 \newcommand{\gztz}{ \parbox[c]{.8in}{\begin{tikzpicture}[scale=.4]
       \coordinate (A) at (0, 0); \coordinate (B) at (2, 0);  \coordinate (C) at (4, 0);
       \draw (C) to[out=135,in=45,loop, min distance=1.5cm] (C);
       \draw[-latex] (A)--(B);       \draw[-latex] (C)--(B);
       \filldraw[black] (A) circle (1.5pt) node[below] {\small 1};
       \filldraw[black] (B) circle (1.5pt) node[below] {\small 2};
       \filldraw[black] (C) circle (1.5pt) node[below] {\small 3};
     \end{tikzpicture}   }}
\newcommand{\gzoo}{ \parbox[c]{.8in}{\begin{tikzpicture}[scale=.4]
       \coordinate (A) at (0, 0); \coordinate (B) at (2, 0);  \coordinate (C) at (4, 0);
       \draw [-latex](C) to[out=135,in=45,loop, min distance=1.5cm] (C);
       \draw[-latex] (A)--(B);       \draw (C)--(B);
       \filldraw[black] (A) circle (1.5pt) node[below] {\small 1};
       \filldraw[black] (B) circle (1.5pt) node[below] {\small 2};
       \filldraw[black] (C) circle (1.5pt) node[below] {\small 3};
     \end{tikzpicture}   }}
 \newcommand{\gzzt}{ \parbox[c]{.8in}{\begin{tikzpicture}[scale=.4]
       \coordinate (A) at (0, 0); \coordinate (B) at (2, 0);  \coordinate (C) at (4, 0);
       \draw [-latex](C) to[out=135,in=45,loop, min distance=1.5cm] (C);
       \draw(A)--(B);       \draw[-latex] (B)--(C);
       \filldraw[black] (A) circle (1.5pt) node[below] {\small 1};
       \filldraw[black] (B) circle (1.5pt) node[below] {\small 2};
       \filldraw[black] (C) circle (1.5pt) node[below] {\small 3};
     \end{tikzpicture}   }}
 \newcommand{\gooo}{ \parbox[c]{.8in}{\begin{tikzpicture}[scale=.4]
       \coordinate (A) at (0, 0); \coordinate (B) at (2, 0);  \coordinate (C) at (4, 0);
       \draw [-latex](C) to[out=135,in=45,loop, min distance=1.5cm] (C);
       \draw[-latex] (B)--(A);       \draw[-latex] (C)--(B);
       \filldraw[black] (A) circle (1.5pt) node[below] {\small 1};
       \filldraw[black] (B) circle (1.5pt) node[below] {\small 2};
       \filldraw[black] (C) circle (1.5pt) node[below] {\small 3};
     \end{tikzpicture}   }}
 \newcommand{\gozt}{ \parbox[c]{.8in}{\begin{tikzpicture}[scale=.4]
       \coordinate (A) at (0, 0); \coordinate (B) at (2, 0);  \coordinate (C) at (4, 0);
       \draw [-latex](C) to[out=135,in=45,loop, min distance=1.5cm] (C);
       \draw[-latex] (B)--(A);       \draw[-latex] (B)--(C);
       \filldraw[black] (A) circle (1.5pt) node[below] {\small 1};
       \filldraw[black] (B) circle (1.5pt) node[below] {\small 2};
       \filldraw[black] (C) circle (1.5pt) node[below] {\small 3};
     \end{tikzpicture}   }}
 \newcommand{\gzto}{ \parbox[c]{.8in}{\begin{tikzpicture}[scale=.4]
       \coordinate (A) at (0, 0); \coordinate (B) at (2, 0);  \coordinate (C) at (4, 0);
       \draw [-latex](C) to[out=135,in=45,loop, min distance=1.5cm] (C);
       \draw[-latex] (A)--(B);       \draw[-latex] (C)--(B);
       \filldraw[black] (A) circle (1.5pt) node[below] {\small 1};
       \filldraw[black] (B) circle (1.5pt) node[below] {\small 2};
       \filldraw[black] (C) circle (1.5pt) node[below] {\small 3};
     \end{tikzpicture}   }}
 \newcommand{\gzot}{ \parbox[c]{.8in}{\begin{tikzpicture}[scale=.4]
       \coordinate (A) at (0, 0); \coordinate (B) at (2, 0); \coordinate (C) at (4, 0);
       \draw [-latex](C) to[out=135,in=45,loop, min distance=1.5cm] (C);
       \draw[-latex] (A)--(B);       \draw[-latex] (B)--(C);
       \filldraw[black] (A) circle (1.5pt) node[below] {\small 1};
       \filldraw[black] (B) circle (1.5pt) node[below] {\small 2};
       \filldraw[black] (C) circle (1.5pt) node[below] {\small 3};
     \end{tikzpicture}   }}

 \begin{tabular}{|l | l |l |c |}
   \hline
   n  & Basis monomial $y^\ba\in (\BZ_G^{e})^{(n)}$& Vector $\ba$ & Orientation $ \Sigma$
\\[2pt]
   \hline
   0 & $1$ & $(0,0,0)$ & \gzzz\\
   \hline
   1 & $ y_1=x_{12}$ & $(1,0,0)$ & \gozz\\
   1 & $ y_2=x_{21}+x_{23}$ & $(0,1,0)$ &  \gzoz\\
   1 & $ y_3=x_{32}+x_{33}$ & $(0,0,1)$ & \gzzo\\
   \hline
   2 & $ y_1y_2=x_{12}x_{23}$ & $(1,1,0)$ & \gooz\\
   2 & $ y_1y_3=x_{12}x_{32}+x_{12}x_{33} $ & $(1,0,1)$ & \gozo\\
   2 & $  y_2^2=2x_{21}x_{23}$ & $(0,2,0)$ & \gztz\\
   2 & $ y_2y_3=x_{21}x_{32}+x_{21}x_{33}+x_{23}x_{33}$ & $(0,1,1)$ & \gzoo\\
   2 & $ y_3^2=2x_{32}x_{33}$ & $(0,0,2)$ &  \gzzt\\
   \hline
   3 & $ y_1y_2y_3=x_{12}x_{23}x_{33}$ & $(1,1,1)$ & \gooo\\
   3 & $ y_1y_3^2 = 2x_{12}x_{32}x_{33}$ & $(1,0,2)$ & \gozt\\
   3 & $ y_2^2y_3=2x_{21}x_{23}x_{33}$ & $(0,2,1)$ & \gzto\\
   3 & $ y_2y_3^2=2x_{21}x_{32}x_{33}$ & $(0,1,2)$ & \gzot\\
   \hline
 \end{tabular}

\end{example}

\

\subsubsection{External bizonotopal algebras distinguish graphs}
\label{sec:isom}
Unlike the usual external zonotopal algebras, which depend only on the cycle matroid of the graph and therefore do not distinguish nonisomorphic graphs with the same matroid (cf.~\cite{Ne}), external bizonotopal algebras are complete graph invariants.

\begin{theorem}\label{thm:isom}
  Let $G_1$ and $G_2$ be graphs without isolated vertices and let $\BZ_{G_1}^e$ and $\BZ_{G_2}^e$ be their external bizonotopal algebras.
  Then the following are equivalent:
  \begin{enumerate}[(1)]
  \item The graphs $G_1$ and $G_2$ are isomorphic.
  \item The $\Zp$-graded algebras $\BZ_{G_1}^e$ and $\BZ_{G_2}^e$ are isomorphic.
  \item The algebras $\BZ_{G_1}^e$ and $\BZ_{G_2}^e$ are isomorphic (as ungraded algebras).
  \end{enumerate}
\end{theorem}

\begin{proof}
  Clearly, if $G_1$ and $G_2$ are isomorphic, then $\BZ_{G_1}^e\simeq \BZ_{G_2}^e$ as graded algebras. The equivalence of (2) and (3) is a known result (see, e.g.,~\cite{BZ}). It remains to prove that (2)$\Rightarrow$(1).

  \medskip

For a nilpotent element $u$ of an algebra, let $\ord(u)$ denote the smallest $m\in \Zp$ such that $u^{m+1}=0$.

If $u\in (\BZ_G^e)^{(1)}$ is a degree one element of the algebra $\BZ_G^e$ of a graph $G=(V,E)$, then Proposition~\ref{prop:partial} implies that $\ord(u)$ is equal to the number of distinct (unoriented) edges $\bar{e}=\{e,e'\}\in E$ whose arrow variables appear with a nonzero coefficient in the expansion of $u$ in the basis $(x_e)$ of $\cEh_G$.

If  $G_1=(V_1,E_1)$ and $G_2=(V_2,E_2)$ are graphs without isolated vertices such that $\BZ_{G_1}^e\simeq  \BZ_{G_2}^e$, then
$$|V_1|=\dim (\BZ_{G_1}^e)^{(1)}= \dim (\BZ_{G_2}^e)^{(1)} = |V_2|.$$

Let $(u_1,u_2,\ldots, u_n)$ be a basis of the space $(\BZ_{G_1}^e)^{(1)}$ minimizing the sum $\ds \sum_{i=1}^n \ord(u_i)$.
We claim that
\[
\ord(u_1)+\ord(u_2)+\ldots+\ord(u_n) =
\dim (\cEh_{G_1})^{(1)}=2|E_1|-\ell_{G_1},
\]
where $\ell_G$ denotes the number of loops in $G$.

Indeed, every arrow variable $x_e$ ($e\in \Eh$) must occur with a nonzero coefficient in the expansion of at least one basis element
$u_i$. Together with~\eqref{eq:oriented} and Proposition~\ref{prop:partial}, this gives the lower bound
$$\ds \sum_i \ord(u_i) \ge 2|E_1|-\ell_{G_1}.$$
On the other hand, for the standard basis $(y_v)_{v\in V_1}$~\eqref{eq:gens} of $(\BZ_{G_1}^e)^{(1)}$ we have
\[   \sum_{v\in V_1} \ord(y_v)=\sum_v \kap_v= 2|E_1|-\ell_{G_1},
\]
so by minimality of $(u_i)$ we have the opposite inequality
\[ \sum_i \ord(u_i)\le 2|E_1|-\ell_{G_1}.\]

Since $u_1,\ldots,u_n$ forms a basis of $(\BZ_{G_1}^e)^{(1)}$, there is an ordering  of the vertex set $V_1=(v_1,\ldots,v_n)$ such that in the expansion
\[u_i=c_{i,1}y_{v_1}+\ldots+c_{i,n}y_{v_n}\]
the diagonal coefficients $c_{i,i}$ are nonzero.
Then $\ord(u_i)\ge \ord(y_{v_i})$ because all edge variables occurring in the expansion of $y_{v_i}$ must also occur in the expansion $u_i$.

 Since
$$
\sum_i \ord(u_i)=\sum_i \ord(y_{v_i})=2|E_1|-\ell_{G_1},$$
we have
$\ord(u_i)=\ord(y_{v_i})$ for all $i$.
Thus the nonzero terms in the expansion of $u_i$ in the variables $x_e$ involve precisely the unoriented edges that appear in the expansion of  $y_{v_i}$, namely the edges incident to $v_i$.
It follows that for a generic $\lambda\in \K$ (i.e.\ for all $\lambda$
outside a finite subset of $\K$) the number of edges between $v_i$ and $v_j$ is equal
\[  \ord(u_i)+\ord(u_j)-\ord(u_i+\lambda u_j).
\]

Knowing $\kappa_{v_i}$ and, for each $j\ne i$, the number of edges between $v_i$ and $v_j$, we can determine the number of loops at  $v_i$.  Therefore  we can reconstruct the graph $G_1$ up to isomorphism from the algebra $\BZ_{G_1}^e$.
Applying the same reconstruction to $\BZ_{G_2}^e$ recovers the graph $G_2$, and hence the isomorphism $\BZ_{G_1}^e\simeq \BZ_{G_2}^e$
implies that the graphs $G_1$ and $G_2$ are isomorphic.
\end{proof}

\subsection{Defining relations of $\BZ_G^e$}
We now describe the relations between the generators of $\BZ_G^e$ given by~\eqref{eq:gens} which, among other things, show that $\BZ_G^e$ is a monomial algebra.

\medskip

Partial score vectors admit a convenient characterization in terms of the degrees $\kap_S$ of subsets of vertices, see~\eqref{eq:degs}.

\begin{proposition}
 \label{prop:scores}
A vector $\ba=(a_v)_{v\in V}\in \Zp^V$ is a partial score vector of a graph $G=(V,E)$ if and only if for every subset $S\subset V$ we have
\begin{equation}
  \label{eq:score-deg}
  \sum_{v\in S}a_v\le \kap_S.
\end{equation}
\end{proposition}
\begin{proof}
  If $\ba=(a_v)$ is a partial score vector corresponding to a partial orientation $\Sigma\subset \Eh$, then~\eqref{eq:score-deg} holds because each arrow $e\in \Sigma$ contributes to exactly one component $a_v$.

 Conversely, given a vector $\ba \in \Zp^V$ satisfying~\eqref{eq:score-deg}, we will construct a partial orientation producing $\ba$ by applying Hall's marriage theorem.
Let $\mathcal{G}$ be a bipartite graph with vertex set consisting of two parts,
$$
A=\{(v,i)\col v\in V, \mathrm{\ such\ that\ } a_v\ne 0, \ i\in\{1,\ldots,a_v\}\}\subset V\times \Zp,
$$
and
$B=E$, the set of edges of $G$. Vertices $(v,i)\in A$ and $e\in B$ are connected by an edge in $\mathcal{G}$ if $v$ is incident to $e$. The inequality~\eqref{eq:score-deg} implies that the graph $\mathcal{G}$ satisfies the condition of Hall's marriage theorem. Therefore there exists a perfect matching $g: A\hookrightarrow B$. This matching produces a partial orientation of $G$ by orienting each edge $e=g(v,i)$ in the image $g(A)$ so that it exits the vertex $v$, and leaving all edges not in $g(A)$ unoriented. The score vector corresponding to this partial orientation is $\ba$.
\end{proof}

\medskip

We denote by $\fM_S$ the set of monomials in $\K[V]:=\K[z_v\col v\in V]$ of total degree $\kap_S+1$ involving only the variables $z_v$ with $v\in S$. In other words,
\begin{equation}
  \label{eq:relat}
  \fM_S=\Bigl\{\prod_{v\in S}z_v^{a_v}\col \sum_{v\in S} a_v=\kap_S+1\Bigr\}.
\end{equation}

\begin{theorem}\label{th:monom}
Given a graph $G=(V,E)$,
 let    $$f_G:\K[V]\to \BZ_G^e, \ z_v\mapsto y_v,$$
be the surjective homomorphism sending free generators of $\K[V]$ to generators~\eqref{eq:gens} of $\BZ_G^e$.
Then $\Ker(f_G)=\fI_G$, where
  \begin{equation}
    \label{eq:ideal}
    \fI_G:=\Big(\bigcup_{ \emptyset \neq S\subset V} \fM_S\Big) \subset  \K[V]
  \end{equation}
is the ideal  generated by the monomials in $\fM_S$ for all nonempty subsets $S\subset V$. In other words, the collection of monomials from $\fM_S$ is the set of defining relations between the generators $y_v$ of the external zonotopal algebra $\BZ_G^e$.
\end{theorem}

\begin{proof} First, we show that for any nonempty subset  $S\subset V$, every monomial $\ds \prod_{v\in S}z_v^{a_v}\in \fM_S$ vanishes in $\BZ_G^e$, i.e.\ belongs to $\Ker f_G$. Indeed, consider the
  expansion of
  \[ f_G\Bigl(\prod_{v\in S}z_v^{a_v}\Bigr)=\prod_{v\in S}y_v^{a_v}\in \cEh_G
  \]
 in arrow variables $x_e$, the generators of the partial orientation algebra $\cEh_G$. Let $x_\Sigma\in \cEh_G$ be some monomial in this expansion.
 By Proposition~\ref{prop:partial} and equation~\eqref{eq:gens}, if $x_\Sigma\ne 0$ then $\Sigma\subset \Eh$ is a partial orientation of $G$ with all arrows $e\in \Sigma$ starting or ending in $S$.
 But  $|\Sigma|=\deg x_\Sigma=\kap_S+1$, which implies that the restriction $\pi|_{\Sigma}$ is not injective. Therefore, $\Sigma$ is not a partial orientation and, hence, every term in the expansion of $\ds \prod_{v\in S}y_v^{a_v}$ vanishes.
This shows that $\fI_G\subset \Ker f_G$ and thus we have a surjection $$\K[V]/\fI_G\twoheadrightarrow \K[V]/\Ker f_G\simeq \BZ_G^e.$$

\medskip

To prove the opposite inclusion $\fI_G\supset \Ker f_G$, it now suffices to show that $f_G$ is injective on the complement $\fI_G^c$ of $\fI_G$, i.e.\ on the set of monomials not contained in $\fI_G$. If  $\mu=\ds \prod_{v\in V}z_v^{a_v}\in \fI_G^c$ is such a monomial, then, by~\eqref{eq:relat}, for every nonempty $S\subset V$ we have $\ds \sum_{v\in S}a_v\le \kap_S$.
By Proposition~\ref{prop:scores}, this means that the exponent vector $\ba=(a_v)_{v\in V}$ of $\mu$ is a partial score vector of $G$ and, by Proposition~\ref{prop:partial},  $\ds f_G(\mu)=\prod_{v\in V}y_v^{a_v}=y^\ba$ is a basis element of $\BZ_G^e$.
Thus the images $f_G(\mu)$ of the monomials in $\fI_G^c$ are linearly independent.
\end{proof}

\subsection{Score vector polytope}
\label{sec:score-polytope}

Proposition~\ref{prop:scores} shows that partial score vectors are the lattice points of a convex polytope with integral vertices.
By Proposition~\ref{lem:score}, these lattice points also label elements in a basis of $\BZ_G^e$.
In this subsection we study the vertices (extreme points) of this polytope.

\begin{definition}
  For a graph $G=(V,E)$ with $n=|V|$ vertices, we define its \emph{score vector polytope} $\Pol_G$ as the set of points in the $n$-dimensional space $\R^V$ satisfying the inequalities~\eqref{eq:score-deg}:
  \begin{equation}
    \label{eq:polytope}
    \Pol_G:=\{(a_v)_{v\in V}\in \R^V_{\ge 0}  \col \sum_{v\in S}a_v\le \kap_S, \mathrm{\ for \ all\ } S\subset V\}.
  \end{equation}
\end{definition}

\medskip
The set of vertices (extreme points) of the polytope $\Pol_G$ can be characterized in several equivalent ways as described in the following theorem.
\begin{theorem}
\label{th:vertices}
For a graph $G=(V,E)$ and a vector
$\ba=(a_v)_{v\in V}\in \Zp^V$, the following conditions are equivalent:
\begin{itemize}
\item[(1)] \
  $\ba$ is a vertex of the score vector polytope $\Pol_G$ of $G$.
\item
[(2)] \ $\ba$ is of the form
$\ba^{\Pi,m}=(a_v)_{v\in V}$, where  $\Pi=(v_1,\ldots,v_n)$ is a linear ordering of $V$, \  $m\in \{0,1,\ldots,n\}$,
and
\begin{equation}
  \label{eq:vertex2}
a_v=\begin{cases}
  0, & \text{if $v=v_i$ for $i\le m$};\\
  \textrm{number of edges between}\ v=v_i \text{ and }  \{v_1,\ldots,v_i\}, & \text{if $i > m$}.
\end{cases}
\end{equation}
\item [  (3)] \
  $\ba$ is of the form
  $\ba^J=(a_v)$,
where $J=(v_1,\ldots,v_r)$ is  a linearly ordered subset of $V$
and
\begin{equation}
\label{eq:vertex3}
  a_v=\begin{cases}
  \textrm{number of edges between}\ v_i \textrm{\ and\ }  V\setminus\{v_1,\ldots,v_{i-1}\}, & \text{if $v=v_i\in J$};\\
    0, & \text{if $v\not\in J$}.
  \end{cases}
\end{equation}
\item [  (4)] \
  $\ba$ is a partial score vector of $G$  corresponding to a unique partial orientation of $G$.
\end{itemize}
\end{theorem}

\ \\
In the proof of this theorem, we will use a special property of the degree function $\kap_S$.
Recall that a function $f:2^V\to \R$ is called \emph{submodular}, if $$f(A)+f(B)\ge f(A\cup B)+f(A\cap B)$$ for all $A,B\subset V$.

\begin{lemma}\label{lem:submodular}  The degree function $\kap$ of any graph $G=(V,E)$ is  submodular.
\end{lemma}
\begin{proof}
To show that
\begin{equation}\label{eq:submodular}
  \kap_{I}+\kap_J\ge \kap_{I\cap J}+\kap_{I\cup J}
\end{equation}
for $I,J\subset V$, we will compare the contributions of a given edge $e\in E$ to both sides of this inequality. There are three possibilities.

(i) \ If $e$ is incident to both $I$ and $J$, then it contributes $2$ to the left-hand side of~\eqref{eq:submodular} and $1$ or $2$ to the right-hand side;

(ii) \ If $e$ is incident to only one of the two subsets $I,J$,
then its contribution to each of the two parts of~\eqref{eq:submodular} is equal to $1$;

(iii) \ If $e$ is incident to neither $I$ nor $J$,
then it contributes $0$ to each part.

Since in each of the three cases the inequality~\eqref{eq:submodular} holds, we conclude that $\kap$ is a submodular function.
\end{proof}

\

\begin{proof}[Proof of Theorem~\ref{th:vertices}]
\ \\
\noindent
\textbf{(1) $\boldsymbol{\Rightarrow}$ (2).}
\
We use induction on the number $n=|V|$ of vertices $G$. If $n=1$, then $G$ has one vertex $v$ and $\ell=|E|$ is the number of edges which are all loops. In this case $\Pol_G$ is the segment $[0,\ell]\subset \Zp$. Its vertices, $0$ and $\ell$, have the required form $\ba^{\Pi,m}$, given by~\eqref{eq:vertex2}, corresponding to the trivial ordering $\Pi=(v)$ and $m=1$ and $m=0$, respectively.

To carry out the induction step, assume that $\ba=(a_v)\in \Zp^V$ is a vertex of $\Pol_G$.
There are two possibilities:
\\[2pt]

\noindent
 (i) \emph{for every nonempty subset $S\subset V$, we have $\ds \sum_{v\in S}a_v< \kap_S$;}
\\[4pt]
\noindent
(ii) \emph{there exists a nonempty subset $S\subset V$ such that
$\ds \sum_{v\in S}a_v=\kap_S$.}
\\[4pt]

In the first case, if $\ba$ has two nonzero components $a_u$ and $a_v$ for  $u\ne v$, then the vectors $\ba'=\ba+\be_v-\be_u$ and $\ba''=\ba-\be_v+\be_u$ are distinct and  belong to $\Pol_G$. (Here $\be_v$ denotes the $v$-th standard basis vector of the lattice $\Z^V$.)
Thus $\ds \ba=\frac{1}{2}(\ba'+\ba'')$ which contradicts the assumption that $\ba$ is a vertex of $\Pol_G$.  If $\ba$ has at most one nonzero component $a_v$, then $\ba=a_v\be_v$. Since both vectors $\ba$  and  $\kap_v\be_v$ belong to $\Pol_G$ and, by our assumption $a_v<\kap_v$, we see that $\ba$ can be a vertex of $\Pol_G$ only when $a_v=0$, i.e.\ $\ba=\mathbf{0}$. This vector has the desired form~\eqref{eq:vertex2}
for $m=n$ and arbitrary $\Pi$.
\\[2pt]

 In the second  case, we claim that there exists $u\in S$ with $a_u=\kap_u$. To prove this, let $S$ be a minimal (by inclusion) subset of $V$ with the above property. We want to show that $|S|=1$.

First notice that if $S'\subset V$ also satisfies $\ds \sum_{v\in S'}a_v=\kap_{S'}$ and $S\cap S' \ne \emptyset$, then $S\subset S'$.
Indeed, if $\emptyset\ne S\cap S'\ne S$, then
 minimality of $S$ implies
$\ds \kap_{S\cap S'}> \sum_{v\in S\cap S'}a_v$. Therefore, by  Lemma~\ref{lem:submodular} we get
$$
\kap_{S\cup S'}\le  \kap_S+\kap_{S'}-\kap_{S\cap S'}<
\sum_{v\in S}a_v+\sum_{v\in S'}a_v - \sum_{v\in S\cap S'}a_v=\sum_{v\in S\cup S'}a_v,
$$
which contradicts the assumption that $\ba\in \Pol_G$.

Now, if $|S|>1$, choose $p,q\in S, \ p\ne q$.
Let us show that the vectors $\ba'=\ba+\be_p-\be_q$ and
$\ba''=\ba-\be_p+\be_q$ belong to
$\Pol_G$. It is enough to do this for $\ba'$.
Since $S$ is minimal by inclusion, we know that  $a_p,\ a_q>0$, which implies that  $\ba', \ \ba''\in \Zp^V$.
For a subset $S'\subset V$, let $\delta_{S'}:V\to \{0,1\}$ be the indicator function of $S'$.
If $|S'\cap\{p,q\}|=0$ or $|S'\cap\{p,q\}|=2$, then $\delta_{S'}(p)=\delta_{S'}(q)$, and we have
$$
\sum_{v\in S'}a_v'=\sum_{v\in S'}a_v+\delta_{S'}(p)-\delta_{S'}(q) = \sum_{v\in S'}a_v \le \kap_{S'}.
$$
If $|S'\cap\{p,q\}|=1$, then $\emptyset\ne S\cap S'\ne S$ and, as we saw above, minimality of $S$ implies that $\ds \sum_{v\in S'}a_v < \kap_{S'}$.
Thus we have
$$
\sum_{v\in S'}a_v'=\sum_{v\in S'}a_v+\delta_{S'}(p)-\delta_{S'}(q) < \kap_{S'} + \delta_{S'}(p)-\delta_{S'}(q) \le \kap_{S'}.
$$
Therefore, $\ba'$ and $\ba''$ are distinct vectors in $\Pol_G$,
which shows that vector $\ds \ba=\frac{1}{2}(\ba'+\ba'')$ cannot be a vertex of $\Pol_G$.

This implies that $|S|=1$, i.e.\ there exists $u\in V$ such that $a_u=\kap_u$. Let $G'=(V',E')$ with $V'=V\setminus\{u\}$ be the graph obtained by removing  from $G$ the vertex $u$ and all edges incident to it.
If we identify $\R^{V'}$ with the affine hyperplane
$H_u:=\{(x_v)\in \R^V \col x_u=\kap_u\}$ in $\R^V$, then the score vector polytope $\Pol_{G'}$ of $G'$ will be identified with the face $\Pol_G\cap H_u$ of the polytope $\Pol_G$. Clearly, vertices of $\Pol_{G}$ lying in $H_u$ correspond to vertices of $\Pol_{G'}$. By the induction hypothesis, the vertex $\ba'\in \Pol_{G'}$ corresponding to  $\ba\in \Pol_G$ is of the form $\ba^{\Pi',m}$~\eqref{eq:vertex2} for some ordering $\Pi'=(v_1,\ldots,v_{n-1})$ of $V'$. Then, appending $u$ at the end of $\Pi'$, we see that the vector $\ba$ is also of the required form, $\ba=\ba^{\Pi,m}$, where $\Pi=(v_1,\ldots,v_{n-1},u)$, thus proving the induction step.

\

\noindent
\textbf{(2) $\boldsymbol{\Rightarrow}$ (1).}
We need to check that every vector $\ba^{\Pi,m}=(a_v)_{v\in V}$ of the form~\eqref{eq:vertex2} is a vertex of $\Pol_G$.
First, from~\eqref{eq:vertex2} we see that $\ds \sum_{v\in S}a_v\le \kap_S$,  for every $S\subset V$, i.e.\ $\ba^{\Pi,m}\in \Pol_G$.

Now assume that $\ba^{\Pi,m}$ is not a vertex of $\Pol_G$, i.e.\ it belongs to the convex hull of some vertices $\bb_1,\ldots,\bb_p\in \Pol_G$.

Let $\preceq$ be the lexicographic order on $\R^V_{\ge 0}$ induced by the reversed ordering $\overline{\Pi}=(v_n,\ldots,v_2,v_1)$ of $V$. That is $\mathbf{x}=(x_{v_1},\ldots,x_{v_n})\preceq  \mathbf{y}=(y_{v_1},\ldots,y_{v_n})$ if $\mathbf{x}-\mathbf{y}=(z_{v_1},\ldots,z_{v_\ell},0,\ldots,0)$, with $z_{v_\ell}<0$, for some $\ell\in [n]$.
Since $\ba^{\Pi,m}$ belongs to the convex hull of $\bb_1,\ldots,\bb_p$, at least one of these vectors must be strictly greater than $\ba^{\Pi,m}$ with respect to $\preceq$.
From~\eqref{eq:vertex2} we know that $a_{v_i}=0$ for $i\le m$, which implies that the first $m$ coordinates of each of $\bb_j\in \R^V_{\ge 0}$ are also $0$. This shows that there is a vector
\begin{equation}
  \label{eq:inpol}
\bb = (0, \ldots, 0, b_{v_{m+1}}, \ldots, b_{v_\ell}, a_{v_{\ell+1}}, \ldots, a_{v_n}) \in   \Pol_G,
\end{equation}
with $b_{v_\ell} > a_{v_\ell}$ for some $\ell \in \{m+1,\ldots,n\}$.
Then, for $S=\{v_\ell,\ldots,v_n\}$ we have
$$ \sum_{v\in S}b_v>\sum_{v\in S} a_v=\sum^n_{i=\ell} a_{v_i}=\kap_S.$$
Hence, $\bb$ does not satisfy~\eqref{eq:polytope} which gives a contradiction with~\eqref{eq:inpol}.

\

\noindent
\textbf{(2) $\boldsymbol{\Leftrightarrow}$ (3).}
\
If a vector $\ba=\ba^{\Pi,m}$  is of the form~\eqref{eq:vertex2}  with an ordering $\Pi=(v_1,\ldots,v_n)$ and $m\in \{0,1,\ldots,n\}$, then $\ba$ can be presented in the form~\eqref{eq:vertex3},  $\ba=\ba^{J}$ by taking $J=(v_n,v_{n-1},\ldots,v_{n-m})$. (In particular, $J=\emptyset$, if $m=n$).
Conversely, if  $\ba=\ba^{J}$ for an ordered subset $J=(v_1,\ldots,v_r)$, then $\ba=\ba^{\Pi,m}$, with $m=n-|J|$ and the ordering $\Pi=(v_r,v_{r-1},\ldots)$ obtained by reversing $J$ and appending to it the complement $V\setminus J$ in any order.

\

\noindent
\textbf{(2) $\boldsymbol{\Rightarrow}$ (4).}
A vector  $\ba^{\Pi,m}$ of the form~\eqref{eq:vertex2}  is
the score vector of the partial orientation in which an edge $e\in E$ between vertices $v_i$ and $v_j$ with $i\ge j$ is oriented from $v_i$ to $v_j$, if $i> m$, and is unoriented otherwise. The uniqueness of such partial orientation follows from the observation that the component $a_{v_i}$ of  $\ba^{\Pi,m}$, for $i=m+1,\ldots,n$, is equal to
$\kap_{\{v_i, v_{i+1},\ldots,v_n\}}-\kap_{\{v_{i+1},\ldots,v_n\}}$.

\

\noindent
\textbf{(4) $\boldsymbol{\Rightarrow}$ (3).}
\
Let $\ba=(a_v)_{v\in V}\in \Pol_G$ be a partial score vector that corresponds to a unique partial orientation $\Sigma\subset \Eh$.
To construct an ordered subset $J\subset V$ such that $\ba=\ba^J$, we will use two special properties of  $\Sigma$.

First, we claim that if an arrow $e \in \Eh$  belongs to $\Sigma$, then every edge incident to the source $v=s(e)\in V$ of $e$ must also be oriented in $\Sigma$ (i.e.\ for every $f\in s^{-1}(v)\subset \Eh$ either $f$ or $f'$ is in $\Sigma$).
Indeed, if neither $f$ nor $f'$ are in $\Sigma$, then after replacing $e$ by $f$ we will obtain a new partial orientation $\Sigma'=(\Sigma\cup \{f\})-\{e\}$ which gives the same partial score vector $\ba$.

Second, we claim that $\Sigma$ contains no oriented cycles other than loops. Indeed, if non-loop arrows $e_1,e_2,\ldots,e_p\in \Sigma$ form an oriented cycle, then replacing them in $\Sigma$ with oppositely oriented arrows $e'_1,\ldots,e'_p$, will not change the partial score vector.

Now, let $J\subset V$ be the set of all vertices $v\in V$ with $a_v>0$. If $J=\emptyset$, then $\ba=\mathbf{0}=\ba^J$. Thus we can assume that $J$ is nonempty.
From the first property it follows that every edge in $G$ incident to some vertex from $J$ is oriented in $\Sigma$. Since $\Sigma$ has no oriented cycles and $J\ne\emptyset$, there exists a source vertex $v_1\in J$ (i.e.\  all edges adjacent to  $v_1$ are oriented away from $v_1$). The induced orientation on  $J-\{v_1\}$ is still acyclic, so we can choose $v_2\in J-\{v_1\}$ so that all edges incident to $v_2$ are oriented away from $v_2$ except possibly those incident to $v_1$.
Continuing in this way, we will obtain an ordering of $J=(v_1,\ldots,v_{|J|})$  such that  for each $i$ all edges incident to $v_i$ are oriented away from $v_i$ except those incident to $\{v_1,v_2,\ldots,v_{i-1}\}$. By construction,  we have $\ba=\ba^J$.
\end{proof}

In general it is difficult to find the exact number of vertices of the score vector polytope $\Pol_G$.  However, we have the following  upper bound.

\begin{corollary}
  For a graph $G$ with $n$ vertices, the polytope $\Pol_G$ has at most $\lfloor e\cdot n! \rfloor$ vertices.

  If $G$ is a simple graph, then the number of vertices of $\Pol_G$ does not exceed $\lfloor (e-1)\cdot n! \rfloor$. This bound is exact  only if $G$ is the complete graph $K_n$.
\end{corollary}

\begin{proof}
  From part (3) of Theorem~\ref{th:vertices} it follows that the number $N_G$ of vertices in $\Pol_G$ is  less than or equal to the number of ordered subsets of $V$, i.e.\
  $$ N_G\le \sum_{S\subset V}(|S|)!=\sum_{m=0}^n \binom{n}{m}m!=
\sum_{m=0}^n \frac{n!}{(n-m)!}=n!\sum_{i=0}^n\frac{1}{i!}=\lfloor e\cdot n!\rfloor.
$$

If $G$ is a simple graph, then any ordering  $J=(v_1,\ldots,v_{n-1},v_n)$  of the full set $V$ and its truncation $J'=(v_1,\ldots,v_{n-1})$ give the same vertex $\ba^J=\ba^{J'}$ of $\Pol_G$. Therefore, in this case, we can drop the last term in the above sum and obtain a better estimate
\begin{equation}
  \label{eq:simple}
  N_G\le n!\sum_{m=0}^{n-1}\frac{1}{(n-m)!}=n!\sum_{i=1}^n\frac{1}{i!}=\lfloor (e-1)\cdot n!\rfloor.
\end{equation}
\\[4pt]

If $G=K_n$ is a complete graph with $V=[n]=\{1,\ldots,n\}$, then
the vertex of $\Pol_G$ corresponding by~\eqref{eq:vertex3} to the ordered subset $J=(1,\ldots,m)\subset [n]$ is given by
$$\ba^J=(n-1,n-2,\ldots,n-m,0,\ldots,0)$$ and all other vertices are obtained from it by permutations of $[n]$.
This vertex has $m$ distinct nonzero entries and, therefore, each of $n!/(n-m)!$ choices of nonempty ordered $m$-element subsets of $V$ gives a distinct vertex of $\Pol_G$. Thus for the complete graph, ~\eqref{eq:simple} becomes an equality:
$$
N_{K_n}=\sum_{i=1}^n \frac{n!}{i!}=\lfloor (e-1)\cdot n!\rfloor.
$$

Finally, for a non-complete simple graph $G$ with $n$ vertices, choose
two vertices $v_1, v_2\in V$ not connected by an edge. Then the ordered subsets $J_1=(v_1,v_2)$ and $J_2=(v_2,v_1)$ give  the same vertex $\ba^{J_1}=\ba^{J_2}$ of $\Pol_G$. Therefore, in this case, the inequality~\eqref{eq:simple} is strict.
\end{proof}

\subsection{Weak  parking functions of a graph}
\label{sec:park}

Here we will give yet another characterization of the score vector polytope $\Pol_G$ of a graph $G=(V,E)$ and, thus, another description of the monomial basis of the external bizonotopal algebra $\BZ_G^e$.

Recall the notion of a   parking function of a graph introduced by Postnikov and Shapiro in~\cite{PS}.
\begin{definition}
Let $G=(V,E)$ be a graph. For a subset of vertices $S\subset V$ and $v\in S$, denote by $d_S(v)$ the number of edges in $E$ connecting $v$ with vertices in $V\setminus S$.

A \emph{$G$-parking function}, relative to a distinguished vertex $q\in V$, is a function $f:V\setminus \{q\}\to \Zp$ such that for each nonempty subset $S\subset V\setminus \{q\}$, there exists a vertex $v\in S$ with
$f(v)< d_S(v)$.
\end{definition}

The above definition does not take into account loops of $G$. To change this,
we introduce a modification of the concept of a $G$-parking function which does not require a choice of a distinguished vertex and, as we will see below, is related to the score vector polytope $\Pol_G$.

\begin{definition}
\label{def:park}
  For a subset of vertices $S\subset V$ of a graph $G=(V,E)$ and a vertex $v\in S$, we denote by $\dh_S(v)$ the number of edges in $E$ with one end at $v$ and the other in $(V\setminus S)\cup \{v\}$.
  In other words,
  $$\dh_S(v)=d_S(v)+\ell(v),$$
  where $\ell(v)$ is the number of loops at $v$. In particular,
  \[ \dh_V(v)=\ell(v) \text{ and } \dh_{v}(v):=\dh_{\{v\}}(v)=\kap_v.
  \]

  A \emph{weak $G$-parking function}  is a function
  $f:V\to \Zp$ such that for each nonempty subset $S\subset V$, there exists a vertex $v\in S$ with $f(v)\le \dh_S(v)$.
\end{definition}

In the definition of a $G$-parking function we had to exclude the distinguished vertex because otherwise the inequality $f(v)<d_S(v)$ would be impossible to satisfy for $S=V$. For weak parking functions this problem does not arise.
However, we can view weak parking functions as parking functions for a special graph.

\begin{definition}
The \emph{delooped cone} of a graph $G=(V,E)$ is
the graph  $C_G$  obtained by adding to $G$ a new vertex, called the \emph{apex}, connected by edges to every vertex of $G$, and replacing each loop by an edge connected to the apex.
More precisely, if $L$ is the set of loops in $G$, then
$$C_G=(V_C,E_C),$$
where
$$
V_C=V\sqcup \{v_0\}
$$
and
$$
 E_C=(E-L)\sqcup \{(v,v_0)\col v\in V\} \sqcup \{(s(\ell),v_0) \col  \ell\in L\}.
 $$

If $G$ is a loopless graph then $C_G$ is just the usual cone graph of $G$.
\end{definition}

\

\begin{example}
  We will illustrate these notions for the following graphs
$G_1=
\parbox[c]{0.6in}{\ \\[10pt]  \begin{tikzpicture}
      [scale=.5]
      \coordinate (A) at (0, 0);
      \coordinate (B) at (2, 0);
      \draw (A) -- (B);
      \filldraw[black] (A) circle (2pt) node[below] {\small 1};
      \filldraw[black] (B) circle (2pt) node[below] {\small 2};
    \end{tikzpicture}}
  $,
  \quad
    $ G_2=
    \parbox{0.6in}{   \begin{tikzpicture}
      [scale=.5]
      \coordinate (A) at (0, 0);
      \coordinate (B) at (2, 0);
      \draw (B) to[out=45,in=135,loop, min distance=1.5cm] (B);
      \filldraw[black] (A) circle (2pt) node[below] {\small 1};
      \filldraw[black] (B) circle (2pt) node[below] {\small 2};
    \end{tikzpicture}
}
$, \quad
  $ G_3=
  \parbox{0.6in}{   \begin{tikzpicture}
      [scale=.5]
      \coordinate (A) at (0, 0);
      \coordinate (B) at (2, 0);
      \draw (A) -- (B);
      \draw (B) to[out=45,in=135,loop, min distance=1.5cm] (B);
      \filldraw[black] (A) circle (2pt) node[below] {\small 1};
      \filldraw[black] (B) circle (2pt) node[below] {\small 2};
    \end{tikzpicture}
  }
  $, \ and \quad
  $ G_4=
  \parbox{0.6in}{   \begin{tikzpicture}
      [scale=.5]
      \coordinate (A) at (0, 0);
      \coordinate (B) at (2, 0);
      \draw (A) to[out=30,in=150] (B);
      \draw (A) to[out=-30,in=-150] (B);
      \filldraw[black] (A) circle (2pt) node[below] {\small 1};
      \filldraw[black] (B) circle (2pt) node[below] {\small 2};
    \end{tikzpicture}
  }.
  $
\\[10pt]
  Their delooped cones are
\\
$C_{G_1}=
\parbox[c]{0.6in}{\ \\[10pt]  \begin{tikzpicture}
      [scale=.5]
      \coordinate (A) at (0, 0);
      \coordinate (B) at (2, 0);
      \coordinate (C) at (1, 1.73);
      \draw (A) -- (B) -- (C) -- (A);
      \filldraw[black] (C) circle (2pt) node[above] {\small A};
      \filldraw[black] (A) circle (2pt) node[below] {\small 1};
      \filldraw[black] (B) circle (2pt) node[below] {\small 2};
    \end{tikzpicture}}
  $,
  \
  $ C_{G_2}=
    \parbox{0.6in}{   \begin{tikzpicture}
      [scale=.5]
      \coordinate (A) at (0, 0);
      \coordinate (B) at (2, 0);
      \coordinate (C) at (1, 1.73);
       \draw (A) -- (C) -- (A);
       \draw (B) to[out=95,in=-35,] (C);
       \draw (B) to[out=145,in=-85,] (C);
      \filldraw[black] (C) circle (2pt) node[above] {\small A};
      \filldraw[black] (A) circle (2pt) node[below] {\small 1};
      \filldraw[black] (B) circle (2pt) node[below] {\small 2};
    \end{tikzpicture}
}
$, \ \ $ C_{G_3}=
  \parbox{0.6in}{   \begin{tikzpicture}
      [scale=.5]
      \coordinate (A) at (0, 0);
      \coordinate (B) at (2, 0);
      \coordinate (C) at (1, 1.73);
      \draw (A) -- (C) -- (A);
      \draw (A) -- (B);
      \draw (B) to[out=95,in=-35,] (C);
      \draw (B) to[out=145,in=-85,] (C);
      \filldraw[black] (C) circle (2pt) node[above] {\small A};
      \filldraw[black] (A) circle (2pt) node[below] {\small 1};
      \filldraw[black] (B) circle (2pt) node[below] {\small 2};
    \end{tikzpicture}
  }
  $, \ and \
  $ C_{G_4}=
  \parbox{0.6in}{   \begin{tikzpicture}
      [scale=.5]
      \coordinate (A) at (0, 0);
      \coordinate (B) at (2, 0);
      \draw (A) -- (C) -- (B);
      \draw (A) to[out=30,in=150] (B);
      \draw (A) to[out=-30,in=-150] (B);
      \filldraw[black] (C) circle (2pt) node[above] {\small A};
      \filldraw[black] (A) circle (2pt) node[below] {\small 1};
      \filldraw[black] (B) circle (2pt) node[below] {\small 2};
    \end{tikzpicture}
  }
  $
  and the lists $(f(1),f(2))$  of values of all their weak parking functions  are given in the following table.
  \medskip
  \begin{center}
  \begin{tabular}{|c|c|c|c|}
    \hline
    $G_1$ & $G_2$  & $G_3$ & $G_4$ \\ \hline
    (0,0)  &     (0,0)  &     (0,0)  &     (0,0)  \\ \hline
    (0,1)  &     (0,1)  &     (0,1)  &     (0,1)  \\ \hline
    (1,0)  &            &     (0,2)  &     (0,2)  \\ \hline
          &            &     (1,0)  &     (1,0)  \\ \hline
          &            &     (1,1)  &     (2,0)  \\ \hline
  \end{tabular}
\end{center}
\end{example}

\

We present below some basic properties of weak parking functions.
\begin{theorem}
\label{thm:wpf}
  Let $G=(V,E)$ be a graph.
  \begin{enumerate}[(i)]
  \item Weak parking functions of $G$ are precisely usual parking functions of the delooped cone $C_G$, relative to the apex $v_0$.
  \item If $f$ is a weak parking function of $G$, then
 $$f(v)\le \kap_v=\dh_{v}(v)$$ for all $v\in V$.
\item  If $f$ is a weak parking function of $G$ and $g:V\to \Zp$ is any function such that $g(v)\le f(v)$ for all $v\in V$, then $g$ is also a weak parking function.
\item Let $\Pi=(v_1,\ldots,v_n)$ be a linear ordering of the set of vertices  $V$. The function
  \begin{equation}
    \label{eq:fpi}
  f^\Pi:V\to \Zp, \ f^\Pi(v_i):=\dh_{\{v_i,\ldots,v_n\}}(v_i),
\end{equation}
  assigning to a vertex $v=v_i\in V$ the number of edges from $v$ to vertices $u=v_j, \text{ with } j\le i$, is a weak parking function.
  \item  For every weak parking function $f$, there exists a linear ordering $\Pi$ of $V$ such that $f(v)\le f^\Pi(v)$ for all $v\in V$.
  \item  For every weak parking function $f$, we have  $\ds \sum_{v\in V}f(v)\le |E|$.
  \item A weak parking function $f$ is maximal with respect to the point-wise order if and only if $$ \sum_{v\in V}f(v) =|E|$$
or, equivalently, when $f=f^\Pi$ for some linear ordering $\Pi$ of $V$.
  \end{enumerate}
\end{theorem}

\begin{proof}
Statements (i)-(iii) follow immediately from the above definitions.

To prove (iv), take a subset $\emptyset\ne S\subset V$ and consider $v=v_m\in S$, where $m:=\min \{i\col v_i\in S\}$.
Then $S\subset S_m:=\{v_m,\ldots,v_n\}$ and therefore $f^\Pi(v)=\dh_{S_m}(v)\le \dh_S(v)$,
which shows that $f^\Pi$ is a weak parking function.

To show (v), let $f$ be a weak parking function and construct a linear
ordering $\Pi=(v_1,\ldots,v_n)$ of $V$ as follows. Start with a vertex $v_1\in V$ such that $f(v_1)\le \dh_V(v_1)=\ell(v_1)$. Proceeding
inductively, if we already have an ordered collection $(v_1,\ldots,v_m)$
with $m<n$, we take $v_{m+1}$ to be a vertex
$v\in S=V\setminus \{v_1,\ldots,v_m\}\ne \emptyset$ such that $f(v)\le \dh_S(v)$. From this construction it is clear that the resulting ordering $\Pi$ satisfies $f(v)\le f^\Pi(v)$ for all $v\in V$.

Statement (vi) now follows from (v), since, by definition of the function $f^\Pi$, it satisfies $\ds \sum_{v\in V}f^\Pi(v)=|E|$.

Finally, (vii) follows from (v) and (vi).
\end{proof}

\

Similar to the case of usual $G$-parking functions, the number of weak parking functions of a loopless graph has a combinatorial interpretation.

\begin{corollary}
The number of weak parking functions of a graph $G$ without loops is equal to the number of rooted spanning forests in $G$ (i.e.\ spanning forests with a distinguished vertex in every component).
\end{corollary}
\begin{proof}
  For a loopless graph $G$ its delooped cone $C_G$ coincides with the usual cone of $G$. Therefore, by part (i) of the above theorem, the number of weak parking functions of $G$ is equal to the number of graph parking functions of $C_G$ relative to its apex.
  By~\cite[Theorem 2.1]{PS}, the latter number is equal to the number of spanning trees in $C_G$. Removing the apex $v_0$  from a spanning tree of $C_G$ turns it into a spanning forest of $G$ with one marked vertex in every component, which clearly gives a bijection between spanning trees of $C_G$ and rooted spanning forests of $G$.
\end{proof}

\

Now we will discuss the connection between weak parking functions and partial score vectors.

\begin{definition}
  Given a weak parking function $f:V\to \Zp$ of a graph $G$,
 the vector
  $$\bff:=(f(v))_{v\in V} \in \Zp^V$$
is called the \emph{parking vector} of $G$ corresponding to $f$.
\end{definition}

\begin{theorem}
  \label{thm:parking} Let $G=(V,E)$ be a graph.
\begin{enumerate}[(i)]
\item Every parking vector $\bff$ of $G$ is a partial score vector.

\item Parking vectors of $G$ are precisely the partial score vectors of $G$ coming from \textbf{acyclic} partial orientations.

  \item Parking vectors corresponding to maximal (with respect to the component-wise order) weak parking functions of $G$ are the partial score vectors coming from acyclic \textbf{total} orientations.

  \item  Every vertex  $\ba = (a_v)_{v\in V}$ of the score vector polytope
    $\Pol_G$ of a graph $G$ is a parking vector, i.e.\
    $f: V\to \Zp \col v\mapsto a_v$ is a weak parking function.
  \item The score vector polytope $\Pol_G$ of $G$ is the convex hull of the set of all parking vectors.
   \end{enumerate}
\end{theorem}
\begin{proof}
 (i) Let $f$ be a weak parking function.
 To construct a partial orientation $\Sigma\subset \Eh$ whose
 score vector is equal to the parking vector $\bff=(f(v))_{v\in V}$,
 consider a linear ordering  $\Pi=(v_1,\ldots,v_n)$  of $V$ such that
$$ f(v_i)\le \dh_{V\setminus \{v_1,\ldots,v_{i-1}\}}(v_i).$$
Existence of such ordering is guaranteed by part (v) of Theorem~\ref{thm:wpf}.
Therefore, for each $i\le n=|V|$, we can find a subset  $E_i\subset E$
with $|E_i|=f(v_i)$ edges connecting vertex $v_i$ with vertices in $\{v_1,\ldots,v_i\}$. By orienting each edge in $E_i$ out of $v_i$, we obtain a subset $\Eh_i$ of $\Eh$. Since $E_i\cap E_j=\emptyset$ for $i\ne j$, the disjoint union $\ds \Sigma=\bigsqcup_i \Eh_i$
  is a partial orientation whose score vector is equal to $\bff$.

  (ii) The partial orientation constructed in the proof of part (i) from a weak parking function $f$ is acyclic because its arrows can only go
from $v_i$ to $v_j$ with $j\le i$.

Conversely, let $\ba=(a_v)_{v\in V}\in \Pol_G$ be a partial score vector corresponding to an acyclic partial orientation   $\Sigma\subset \Eh$.
Since $\Sigma$ is acyclic, it induces a partial order $\preceq_\Sigma$ on $V$, namely $u \preceq_\Sigma v $ when there is an oriented path from $v$ to $u$ formed by arrows in $\Sigma$. Let $\Pi$ be a linear ordering extending $\preceq_\Sigma$ and let $f^\Pi$ be the corresponding weak parking function
given by~\eqref{eq:fpi}. Then
\[a_v \le |\{u\in V\col u \preceq_\Sigma v\}|\le f^\Pi(v)
\]
for all $v\in V$. By part (iii) of Theorem~\ref{thm:wpf}, we conclude that $\ba$ is a parking vector.

  (iii) If $f:V\to \Zp$ is a weak parking function whose vector $\bff=(f(v))$ is the score vector of a partial orientation $\Sigma$,
  then $\ds \sum_{v\in V}f(v)=|\Sigma|$. By part (vii) of Theorem~\ref{thm:wpf}, $f$ is maximal exactly when $\ds \sum_{v\in V}f(v)=|E|$
  or, equivalently, when $|\Sigma|=|E|$ i.e.\ when $\Sigma$ is a total orientation.

  (iv) If $\ba=(a_v)_{v\in V}$ is a vertex of $\Pol_G$, then by part (4) of Theorem~\ref{th:vertices} it corresponds to a unique partial orientation $\Sigma\subset \Eh$. If $\Sigma$ had an oriented cycle $(e_1,e_2,\ldots,e_p)$, then, by reversing orientations of all the arrows $e_i, \ i=1,\ldots, p$, we would obtain a different partial orientation
  \[\Sigma'=(\Sigma-\{e_1,\ldots,e_p\})\cup \{e'_1,\ldots,e'_p\}
  \]
  giving the same partial score vector $\ba$, which contradicts uniqueness of $\Sigma$. Therefore, orientation $\Sigma$ is acyclic and, by (ii), $\ba$ is a parking vector.

  (v) Since $\Pol_G$ is a convex polytope, it is a convex hull of the set of its vertices. By (iv), each vertex of  $\Pol_G$ is a parking vector and, by (i), every parking vector of $G$ belongs to $\Pol_G$. This shows that $\Pol_G$ is the convex hull of the set of parking vectors.
\end{proof}

\

The following result is an immediate consequence of part (v) of the above theorem and Proposition~\ref{prop:scores}.
 \begin{corollary}
 The dimension of the external bizonotopal algebra $\BZ_G^e$ of a graph $G$ is equal to the number of lattice points in the convex hull of the set of parking vectors of $G$.
\end{corollary}
\qed

In particular, for a complete graph $K_n$ on $n$ vertices,
$\dim \BZ^e_{K_n}$ is equal to the number of lattice points in the \emph{parking functions polytope} $\Pol_n$ which was studied in several recent papers (cf.~\cite{AW} and~\cite{HLVM}).

\section{$r$-bizonotopal algebras and loopy deletion-contraction}\label{sec:r-bizon}

\subsection{Definition of $r$-bizonotopal algebras}\label{ssec:r-bi}

Here we consider bizonotopal analogs of central and internal zonotopal algebras (see~\cite{HoRo}). They are members of a more general family which we call $r$-bizonotopal and introduce below.

\medskip

\noindent
Let
\[\delta_G:=\min_{v\in V}\kap_v\]
be the smallest number of edges incident to a vertex of $G$.
(Recall that loops are only counted once in $\kap_v$, and therefore, in general, $\delta_G$ is not the same as the minimal degree of a vertex in $G$.)

For $r\in \Z$ and a subset $S\subset V$ consider, similarly to~\eqref{eq:relat}, the set of monomials in the polynomial ring $\K[V]:=\K[z_v\col v\in V]$ given by
  \begin{equation}
    \label{eq:r-relat}
    \fM^{(r)}_S=\Bigl\{\prod_{v\in S}z_v^{a_v}\col \sum_v a_v=\kap_S+r\Bigr\}
  \end{equation}
  in variables corresponding to vertices $v\in S$ and of total degree $\kap_S+r$.

  \begin{definition}
    \label{def:r-bizonot}
Let $G=(V,E)$ be a graph and fix an integer $r\ge -\delta_G$.
The $r$-\emph{bizonotopal algebra} of $G$ is the quotient algebra
  \begin{equation}
    \label{eq:r-alg}
    \BZ_G^{(r)}:=\K[V]/\fI_G^{(r)},
  \end{equation}
where
  \begin{equation}
    \label{eq:r-ideal}
    \fI_G^{(r)}:=\Big( \bigcup_{\emptyset \ne S\subset V}\fM_S^{(r)}  \Big)\subset  \K[V]
  \end{equation}
  is the ideal  generated by the monomials from $\fM_S^{(r)}$ for all nonempty subsets $S\subset V$.
\

\medskip

According to Theorem~\ref{th:monom}, the algebra  $\BZ_G^{(1)}$ is isomorphic
to the external bizonotopal algebra $\BZ_G^e$. By analogy with the usual zonotopal algebras~\cite{HoRo}, we call the algebras
\begin{equation}
  \label{eq:central}
  \BZ_G^c:=  \BZ_G^{(0)}
\end{equation}
and
\begin{equation}
  \label{eq:internal}
  \BZ_G^i:=  \BZ_G^{(-1)}
\end{equation}
the \emph{central} and \emph{internal} bizonotopal algebras of $G$  respectively. We will discuss their special properties in more detail in Sections~\ref{subsec:ce} and~\ref{sec:in}.

\medskip

Algebras $\BZ_G^{(r)}$ corresponding to $r>1$ are called \emph{superexternal} and the ones with $r<-1$ are called \emph {subinternal}.
\end{definition}

\medskip

In particular, for a graph $G$ with one vertex and $\ell$ loops, the $r$-bizonotopal algebra
is isomorphic to a truncated polynomial algebra
\begin{equation}
  \label{eq:onevert}
  \BZ_G^{(r)}\simeq \K[z]/(z^{\ell+r}).
\end{equation}

\medskip

The next result extends Theorem~\ref{thm:isom} to show that superexternal bizonotopal algebras are complete graph invariants without the need to exclude isolated vertices.

\begin{theorem}\label{thm:isom:super}
For any graphs $G_1=(V_1,E_1)$ and $G_2=(V_2,E_2)$ and integer $r\ge 2$, the superexternal algebras $\BZ_{G_1}^{(r)}$ and $\BZ_{G_2}^{(r)}$ are isomorphic if and only if the graphs $G_1$ and $G_2$ are isomorphic.
\end{theorem}

\begin{proof}
  For isomorphic graphs $G_1$ and $G_2$, the algebras $\BZ_{G_1}^{(r)}$ and $\BZ_{G_2}^{(r)}$ are clearly isomorphic, so we only need to prove the converse statement.

  Assume that $\BZ_{G_1}^{(r)}\simeq\BZ_{G_2}^{(r)}$. Then, as explained in the beginning of the proof of Theorem~\ref{thm:isom}, these algebras are isomorphic as graded algebras and we have
  \[ \dim (\BZ_{G_1}^{(r)})^{(1)}=\dim (\BZ_{G_2}^{(r)})^{(1)}.\]
  Let $\bar{I}=(\bar{z}_v)_{v\in V_1}\subset \BZ_{G_1}^{(r)}$ be the image of the ideal
  $I=(z_v)_{v\in V_1}\subset \K[V_1]$.
  It is the maximal ideal of the Artinian local ring $\BZ_{G_1}^{(r)}$. Since $r\ge 2$, the relation ideal $\fI_{G_1}^{(r)}\subset \K[V_1]$ is contained in $I^2$. This impies that
  \[\dim (\BZ_{G_1}^{(r)})^{(1)} = \dim \bar{I}/(\bar{I})^2=|V_1|.\] Thus, the number of vertices of $G_1$ including isolated vertices can be recovered from the algebra $\BZ_{G_1}^{(r)}$. Now the rest of the proof proceeds similar to the proof of Theorem~\ref{thm:isom}.
 \end{proof}

\subsection{Loopy deletion-contraction}
\label{sec:del-contr}

In this subsection, we show that the Hilbert series of the central, external, and superexternal bizonotopal algebras satisfy a deletion--contraction relation analogous to the classical one, allowing them to be computed recursively.

\begin{definition}
Given a graph $G=(V,E)$ and a non-loop edge $e\in E$, we consider two operations, \emph{deletion} and \emph{loopy contraction}, producing two new graphs:
\begin{itemize}
\item
$G-e$, the graph with the same vertex set as $G$ and with the edge $e$ removed from the edge set;
\item
$G/e$, the graph obtained from $G$ by identifying the endpoints of the edge $e$ \emph{without deleting it}, so that the edge $e$, as well as all other edges connecting its endpoints, become loops.
\end{itemize}
\end{definition}

\begin{remark}
  Our loopy contraction operation differs from the contraction operation familiar from the theory of the Tutte polynomial and its variants.
  It does not remove any edges; instead, it turns them into loops, so terms such as ``looping'' or ``loopification'' might be more descriptive.
  Nevertheless, we retain the traditional terminology to emphasize the formal similarity of our construction with the Tutte deletion--contraction framework.

  As we will show elsewhere~\cite{KNSV}, the loopy deletion--contraction relation~\eqref{eq:del-contr} gives rise to a new multivariable graph polynomial whose structural properties are analogous to those of the Tutte polynomial and Stanley's chromatic symmetric function, while being inequivalent to either.

\end{remark}

For a graph $G$, let $h^r(t)$ denote the Hilbert series of its $r$-bizonotopal algebra~\eqref{eq:r-alg}, that is, the generating function of the dimensions of the homogeneous components of $\BZ_G^{(r)}$:
\[
h^r(t)=\sum_{n\ge 0}\dim (\BZ_G^{(r)})^{(n)}\, t^n.
\]

\begin{theorem}\label{thm:delcont}
For $r\ge 0$ and a non-loop edge $e\in E$ of a graph $G=(V,E)$, the Hilbert series $h^r(t)$ of the $r$-bizonotopal algebras of the graphs $G$, $G-e$, and $G/e$ satisfy the following \textbf{loopy deletion--contraction relation}:
\begin{equation}\label{eq:del-contr}
  h_G^r(t)=h_{G/e}^r(t)+t\cdot h_{G-e}^r(t).
\end{equation}
\end{theorem}
\begin{proof}
By Definition~\ref{def:r-bizonot}, the $r$-bizonotopal algebra $\BZ_G^{(r)}$ has a homogeneous basis consisting of monomials
\begin{equation}
  \label{eq:monoms}
  \mu=\prod_{v\in V} z_v^{d_v}\in \K[V]
  \end{equation}
satisfying
\begin{equation}
  \label{eq:bounds}
  \sum_{v\in I}d_v< \kap_I+r
\end{equation}
for all nonempty $I \subset V$. We call such monomials \emph{admissible}.

\medskip
We first treat the simplest case of a two-vertex graph $G$, to which the general case will later be reduced.
Let $G$ be a graph with vertices $p$ and $q$ having $\ell_p$ and $\ell_q$ loops respectively and $m$ edges from $p$ to $q$, for example
\quad
$$G= \parbox[c]{1.5in}{  \begin{tikzpicture}[scale=.9]
    \coordinate (A) at (0, 0);
    \coordinate (B) at (2, 0);
% edges
    \draw (A) -- (B);
    \draw (A) to[out=30,in=150] (B);
    \draw (A) to[out=-30,in=-150] (B);
% loops
    \draw (A) to[out=45,in=135,loop, min distance=1.5cm] (A);
    \draw (B) to[out=45,in=135,loop, min distance=1.5cm] (B);
    \draw (B) to[out=-45,in=-135,loop, min distance=1.5cm] (B);
%vertices
    \filldraw[black] (A) circle (2pt) node[left] {\small p};
    \filldraw[black] (B) circle (2pt) node[right] {\small q};
  \end{tikzpicture}
}
  $$
  The degrees $\kap_I$ for the nonempty subsets of $V$ are
  $$ \kap_p=\ell_p+m, \ \kap_q=\ell_q+m, \text{ and } \kap_{\{p,q\}}=\ell_p+\ell_q+m.
  $$
  Using the notation $s=m+r-1$, the inequalities~\eqref{eq:bounds} show that the admissible monomials $z_p^xz_q^y $ of the algebra $\BZ^{(r)}_G$ correspond to integer points $(x,y)$ in the polygonal region of the plane
\begin{equation}
  \label{eq:region}
  R = \{(x,y)\in \R^2\col 0\le x\le \ell_p+s, \ 0\le y\le \ell_q+s, \
  x+y\le \ell_p+\ell_q+s\}
\end{equation}
shown on the figure below. (We know that the line $x+y=\ell_p+\ell_q+s$ intersects the region $x\le \ell_p+s, y\le \ell_q+s$ because $s\ge 0$.)

\medskip

\begin{tikzpicture}[scale=0.8],every node/.style={font=\small}]
  % shifted smaller region
  \fill[pattern=north east lines,pattern color=black!30]
  (0,1) -- (5,1) -- (5,4) -- (2,7) -- (0,7) -- cycle;
  % axes
  \draw[->] (-0.2,0) -- (7.2,0) node[right, font=\large] {$x$};
  \draw[->] (0,-0.2) -- (0,8.2) node[left, font=\large] {$y$};
\node[below left, font=\large] at (0,0) {$0$};
\node[left,color=orange,font=\large] at (0,7) {$R$};
\node[below left,font=\large, color=blue] at (0,6) {$R'$};
\node[above left,font=\large] at (5,1) {$R''$};
% bounding lines: x=6, y=7, x+y=9
  \draw[dashed] (6,0) -- (6,3.8) node[right] {$x=\ell_p+s$};
  \draw[dashed] (0,7) -- (3.2,7) node[right] {$y=\ell_q+s$};
  \draw[dashed] (1.5,7.5) -- (6.5,2.5) node[right] {$x+y=\ell_p+\ell_q+s$};
  % bounding lines: x=6, y=7, x+y=9
  \draw[dotted] (5,0) -- (5,4.9) node[right] {$x=\ell_p+s-1$};
  \draw[dotted] (0,6) -- (3.9,6) node[ right ] {$y=\ell_q+s-1$};
  \draw[dotted] (0.3,7.7) -- (6.7,1.3) node[right ] {$x+y=\ell_p+\ell_q+s-1$};
%  boundaries
  \draw[very thick,color=orange] (0,0) -- (6,0) -- (6,3) -- (2,7) -- (0,7) -- cycle;
  \draw[thick](0,1) -- (5,1) -- (5,4);
  % lattice points satisfying: x,y>=0, x<=6, y<=7, x+y<9
  \draw[very thick,color=blue] (5,0) -- (5,3) -- (2,6) -- (0,6) --(0,0)--cycle;
  \foreach \x in {0,...,5} {
    \fill[color=red] (\x,0) circle (2.2pt);
    \foreach \y in {1,...,7} {
      \pgfmathtruncatemacro{\s}{\x+\y}
      \ifnum\s<10
        \fill (\x,\y) circle (2.2pt);
      \fi
    }
  }
  \foreach \y in {0,...,3} { \fill[color=red] (6,\y) circle (2.2pt);}
\end{tikzpicture}

\medskip

Since for the graph $G'=G-e$ we have $m'=m-1$ and so $s'=s-1$,
the admissible monomials for $\BZ_{G'}$ correspond to integer points in the region $R'$ obtained from $R$ by shifting the bounding lines $x=\ell_p+s, \ y=\ell_q+s$ and $x+y=\ell_p+\ell_q+s$ by one unit towards the origin.
Thus, multiplication by $z_q$ sends an admissible monomial $\mu=z_p^xz_q^y$ for $\BZ_{G'}$ to an admissible monomial $z_p^xz_q^{y+1}$ for $\BZ_{G}$ giving a linear embedding
$$j:\BZ_{G'}\hookrightarrow \BZ_{G}$$ of degree $+1$. Geometrically, this map corresponds to shifting $R'$ upwards by $1$ producing the shaded area $R''$ inside $R$.
Consequently, the quotient space $\BZ_{G}/j(\BZ_{G'})$ has a homogeneous basis consisting of the monomials
\begin{equation}
  \label{eq:quot-basis}
  1,z_p,\ldots,z_p^{\ell_p+s},z_p^{\ell_p+s}z_q, \ldots, z_p^{\ell_p+s}z_q^{\ell_q}
\end{equation}
which correspond to the lattice points in $R\setminus R''$ (shown as red dots on the diagram).
Among these $\ell_p+\ell_q+s+1$ monomials there is exactly one monomial of each total degree between $0$ and $\ell_p+\ell_q+s$.

On the other hand,  $G/e$ is a one-vertex graph with $\ell_p+\ell_q+m$ loops. Thus by~\eqref{eq:onevert} we have
$$
\BZ^{(r)}_{G/e}\simeq \K[z]/(z^{\ell_p+\ell_q+m+r})=\K[z]/(z^{\ell_p+\ell_q+s+1}).
$$
Therefore the algebra $\BZ^{(r)}_{G/e} $ has the same Hilbert series as  the graded space $\BZ_{G}/j(\BZ_{G'})$.
This proves the equality~\eqref{eq:del-contr} for graphs with two vertices.

\medskip

The preceding argument yields the following elementary counting statement, which we will use below in the proof of the general case.

Given non-negative integers $a, b, c$ such that $c\le a+b$, let $N(a,b,c)$ be the number of points $(x,y)\in \Z^2$ satisfying $0\le x\le a, 0\le y\le b, x+y=c$.
Then
\begin{equation}
  \label{eq:recursion}
  N(a,b,c)-N(a-1,b-1,c-1)=1.
\end{equation}

\medskip

Now let $G$ be an arbitrary graph with $|V|>2$ vertices, and let $p,q\in V$ be the endpoints of the non-loop edge $e$. Denote by $w$ the vertex of the graph $G/e$ obtained by identifying $p$ and $q$.
To prove the Hilbert series relation~\eqref{eq:del-contr} for the graphs $G, \ G-e$, and $G/e$,
we compare admissible monomials for the algebras $\BZ^{(r)}_G , \ \BZ^{(r)}_{G-e}$, and $\BZ^{(r)}_{G/e}$, degree by degree.

For every subset of vertices $I\subset V$ not containing $p$ or $q$, the degrees $\kap_I$  are the same for each of the  graphs $G, \ G-e$, and $G/e$. Therefore the sets of admissible monomials~\eqref{eq:monoms} with $d_p=d_q=0$, i.e.\ monomials  of the kind
\begin{equation}
  \label{eq:mon-fixed}
  \mu=\prod_{v\in V\setminus\{p,q\}} z_v^{d_v},
\end{equation}
for these three graphs coincide.

If we fix the exponents $d_v$ of an admissible monomial~\eqref{eq:mon-fixed} for all vertices $v\in V\setminus \{p,q\}$, then we see from~\eqref{eq:bounds} that the monomial
\[
\widetilde{\mu}=\mu\, z_p^{d_p}z_q^{d_q}
\]
is admissible for $\BZ^{(r)}_G$ if and only if the exponents $d_p$ and $d_q$ satisfy the inequalities
\begin{equation}
  \label{eq:ineqs}
  0\le d_p\le a,\qquad 0\le d_q\le b,\qquad d_p+d_q\le c,
\end{equation}
where
\begin{equation}
  \label{eq:abc}
  \begin{split}
    a&=\min_{I\subset V\setminus\{p,q\}} \kap_{I\cup\{p\}}+r-1-\sum_{v\in I} d_v\,,    \\
    b&=\min_{I\subset V\setminus\{p,q\}} \kap_{I\cup\{q\}}+r-1-\sum_{v\in I} d_v\,,\\
    c&=\min_{I\subset V\setminus\{p,q\}} \kap_{I\cup\{p,q\}}+r-1-\sum_{v\in I} d_v\,.
  \end{split}
\end{equation}
Thus, for fixed $(d_v)_{v\in V\setminus \{p,q\}}$, admissible monomials correspond to integer points in the planer polygon cut out by the inequalities~\eqref{eq:ineqs}.

Deleting the edge $e$ decreases each of the degrees
$\kap_{I\cup\{p\}}$, $\kap_{I\cup\{q\}}$, and $\kap_{I\cup\{p,q\}}$ by one.
Consequently, the monomial $\widetilde \mu$ is admissible for  $\BZ_{G-e}^{(r)}$ if and only if
\begin{equation*}
  \label{eq:ineqs-del}
  0\le d_p\le a-1,\qquad 0\le d_q\le b-1,\qquad d_p+d_q\le c-1.
\end{equation*}
Geometrically, the admissible region for $G-e$ is obtained from that for $G$ by shifting each bounding hyperplane by one unit towards the origin.

For the contracted graph $G/e$, the variables $z_p$ and $z_q$ are replaced by a single variable $z_w$, and the monomial
$\mu\, z_w^{d_w}$
is admissible for $\BZ_{G/e}^{(r)}$ if and only if
$  0\le d_w\le c.$

Thus to compare the Hilbert series and prove~\eqref{eq:del-contr} it suffices to show that for each $c'\in \{0,1,\ldots,c\}$,
the number of integer points $(d_p,d_q)$ satisfying
\begin{equation*}\label{eq:ineq1}
  0\le d_p\le a,\qquad 0\le d_q\le b,\qquad d_p+d_q=c',
\end{equation*}
exceeds by one the number of integer solutions of
\begin{equation*}\label{eq:ineq2}
  0\le d_p\le a-1,\qquad 0\le d_q\le b-1,\qquad d_p+d_q=c'-1.
\end{equation*}

This will follow from~\eqref{eq:recursion} when we establish that $c\le a+b$.
Indeed, let $J_p$ and $J_q$ be subsets of $V\setminus\{p,q\}$ on which the minima defining $a$ and $b$ in~\eqref{eq:abc} are attained.
By Lemma~\ref{lem:submodular}, the function $\kap$  for the graph $G-e$ is submodular.
Therefore,
\[
c\le \kap_{J_p\cup J_q}\le \kap_{J_p}+\kap_{J_q} \le a+b,
\]
which proves the required inequality.
\end{proof}

Theorem~\ref{thm:delcont} has the following consequence for the  external and the central algebras (i.e.\ for $r=1$ and $r=0$, resp.)

\begin{theorem}\label{delcont:ext-cen}
 Let $h_G(t)$ denote the Hilbert series of the central ($r=0$) or external ($r=1$) bizonotopal algebra $\BZ_G^r$ of a graph $G$.  As a function on graphs $h_G(t)$ is uniquely characterized by the following properties:
  \begin{enumerate}[(i)]
  \item {\sf loopy deletion-contraction}:
$$
  h_G(t)=h_{G/e}(t)+t\cdot h_{G-e}(t),
$$
if $e$ is a non-loop edge of $G$;
\item {\sf multiplicativity}:
  \begin{equation}
    \label{eq:multiplic}
   h_{G_1 \sqcup G_2}(t)=h_{G_1}(t) \cdot h_{G_2}(t),
 \end{equation}
\item {\sf initial conditions}:
  \begin{equation}
    \label{eq:n-loops}
  h_{L_n}(t)=
  \begin{cases}\ds
    1+t+\dots + t^{n-1}=\frac{1-t^n}{1-t} ,  & \text{ if } r=0,  \\
\\
    \ds    1+t+\dots + t^n=\frac{1-t^{n+1}}{1-t} , & \text{ if } r=1,
\end{cases}
\end{equation}
where $L_n$ is a one-vertex graph with $n$ loops.
\end{enumerate}
\end{theorem}
\begin{proof}
The first part is a special case of Theorem~\ref{thm:delcont} which establishes the loopy deletion-contraction relation~\eqref{eq:del-contr} for $h_G^r(t)$ for $r\ge 0$.
It is easy to see that the multiplicative property~\eqref{eq:multiplic} holds for $h_G^r$ for all $r\le 1$.
Thus, for $r=0,1$, the computation of $h_G^r(t)$ reduces to using its values on the $n$-loop graph $L_n$, which are given by~\eqref{eq:n-loops}.
\end{proof}

\subsection{``Categorification'' of the deletion-contraction relation}
We now present a purely algebraic proof of the loopy deletion--contraction relation~\eqref{eq:del-contr} for external bizonotopal algebras, modeled on the proof of the usual deletion--contraction relation for the Hilbert series of classical zonotopal algebras given in~\cite[Theorem~2.7]{SSV}.
Unlike the combinatorial argument of the previous section, this proof relies on the realization of the external bizonotopal algebra as a subalgebra of the edge algebra and on functorial maps between these algebras.
In particular, this approach applies only to the external case.

While the proof in the previous section makes the combinatorial content of~\eqref{eq:del-contr} explicit, the argument below is based on the functorial properties of bizonotopal algebras and on a short exact sequence relating the algebras of the graphs $G$, $G-e$, and $G/e$.

\medskip

For a graph $G=(V,E)$ and a non-loop edge $e\in E$ with endpoints $p,q\in V$,
let $\varepsilon', \varepsilon''\in \Eh$ be the arrows corresponding to the orientations of $e$ from $p$ to $q$ and from $q$ to $p$ respectively, let $w$ be the vertex in $G/e$ obtained by identifying vertices $p$ and $q$, and let $\bar{e}$ be the loop  based at $w$ in  $G/e$ and obtained from $e$.

\medskip

There are several natural maps connecting algebras related to the graphs $G$, $G-e$ and $G/e$.
\begin{itemize}
\item
  A surjective homomorphism
$
\widehat{\rho}_e:  \cEh_G\twoheadrightarrow \cEh_{G-e},$
  sending
  $ x_{\varepsilon'}$ and $x_{\varepsilon''}$ to $0$,
  maps $y_v\in \cEh_G$ to $y_v\in \cEh_{G-e}$, for each $v\in V$, and thus
  induces an epimorphism
  \begin{equation}
\label{eq:rho}
\rho_e:\BZ^e_G\twoheadrightarrow \BZ^e_{G-e}.
  \end{equation}

\item The homomorphism
  $ \widehat{\gamma}_e:\cEh_{G/e}\to \cEh_G$ given by \
  $\widehat{\gamma}_e(x_{\bar{e}})= x_{\varepsilon'}+x_{\varepsilon''}$, and $
  \widehat{\gamma}_e(x_\epsilon)=x_\epsilon$,  if  $\epsilon\ne \bar{e}$
  induces an algebra embedding
\begin{equation}
  \label{eq:gamma}
  \gamma_e: \BZ^e_{G/e}\hookrightarrow \BZ^e_{G}: \ \gamma_e(y_w)= y_p+y_q, \text{ and } \gamma_e(y_v)=y_v, \text{ if } v\ne w.
\end{equation}

\item
  The composition of the ``partial derivatives''
  $\ds \frac{\partial}{\partial x_{\varepsilon'}}$
  and   $\ds \frac{\partial}{\partial x_{\varepsilon''}}$
  of $\cEh_G$ with the projection
  $\widehat{\rho}_e:\cEh_G\twoheadrightarrow \cEh_{G-e}\simeq \cEh_G/(x_{\varepsilon'},x_{\varepsilon''})$, gives well-defined derivations
  $ \partial_{x_{\varepsilon'}}, \partial_{x_{\varepsilon''}}: \cEh_G\to \cEh_{G-e}$. Then the map
  $$
  \widehat{\delta}_e = \partial_{x_{\varepsilon'}}- \partial_{x_{\varepsilon''}}:
  \cEh_G\to \cEh_{G-e}
  $$
  is also a derivation of degree $-1$. Moreover, since
  $$\widehat{\delta}_e(y_v)=
  \begin{cases}
    y_v & \text{ if }  v=p\\
    -y_v & \text{ if }  v=q\\
    0 & \text{ if } v\ne p,q,
  \end{cases}
  $$
  we see that $\widehat{\delta}_e$ sends the subalgebra $\BZ_G^e\subset \cEh_G$ onto $\BZ_{G-e}^e\subset \cEh_{G-e}$.
\end{itemize}
It is straightforward to verify that the kernel of the map $\widehat{\delta}_e$
 coincides with the image of the embedding
 $\widehat{\gamma}_e: \cEh_{G/e}\to \cEh_G$ and, if
 $\delta_e: \BZ_G^e\to \BZ_{G-e}^e$ is the restriction
 of $\widehat{\delta}_e$ to $\BZ_G^e$, then
 $\Ker \delta_e=\gamma_e(\BZ_{G/e}^e)$.

\medskip
Thus we established the following proposition.

\begin{proposition}
  \label{thm:delta}
  There exists a surjective, degree $-1$, derivation of graded algebras
  $$\delta_e:\BZ_G^e\twoheadrightarrow \BZ_{G-e}^e$$
whose kernel coincides with the image
$\gamma_e(\BZ_{G/e}^e)\subset \BZ_G^e$ of the embedding $\gamma_e$~\eqref{eq:gamma}.
\end{proposition}

As an immediate consequence of this result we obtain the loopy deletion-contraction relation for external bizonotopal algebras.

\begin{corollary}
 \label{cor:del-cont}
 The Hilbert series of the graded algebras  $\BZ^e_G$, $\BZ^e_{G/e}$, and $\BZ^e_{G-e}$
 satisfy the relation
 \begin{equation}
   \label{eq:hilb}
   h_G(t)=h_{G/e}(t) + t\cdot h_{G-e}(t).
 \end{equation}
\end{corollary}
\begin{proof}
From the above proposition it follows that there exists a short exact sequence of graded vectors spaces
  \begin{equation}
    \label{eq:SES}
    0\to \BZ^e_{G/e} \to \BZ^e_G \to \BZ^e_{G-e}[-1]\to 0,
  \end{equation}
  where $[-1]$ is a degree shift functor indicating
  that the map $\BZ^e_G\to \BZ^e_{G-e}$ decreases degree by $-1$.

Now the relation~\eqref{eq:hilb} follows from
the fact that Hilbert series are additive with respect to exact sequences and that the shift of grading by $-1$ multiplies the Hilbert series by $t$.
\end{proof}

\section{Additional properties of central and internal  bizonotopal algebras}\label{sec:add}
Here we present some additional properties of the central~\eqref{eq:central} and internal~\eqref{eq:internal} bizonotopal algebras of a graph $G$.
In particular, we show that, like the external zonotopal algebra $\BZ_G^e$, they can be viewed as subalgebras of certain quotients of the partial orientation algebra $\cEh_G$~\eqref{eq:partial}.

\subsection{Central bizonotopal algebras}\label{subsec:ce}
Let $G=(V,E)$ be a graph.
We begin with a description of the central bizonotopal algebra $\BZ_G^c=\BZ_G^{(0)}$ as a subalgebra,
analogously to the description of $\BZ_G^e$ in Definition~\ref{def:external}.
First, we define an algebra $\cEh_G^c$, an analog of the partial orientation algebra $\cEh_G$~\eqref{eq:partial}, in the central case.

\medskip

 For a subset of vertices $S\subset V$, let $E_S\subset E$ denote the set of edges incident to vertices in $S$.
 Consider the set
 \[
   \cP^c_S
   :=
   \Bigl\{
   \Sigma\subset    \pi^{-1}(E_S)
   \,\col\,
   \pi|_{\Sigma}:\Sigma\to E_S \text{ is a bijection and } s(\Sigma)\subset S
   \Bigr\}
 \]
of partial orientations whose set of edges is exactly $E_S$ and such that every edge between $S$ and $V\setminus S$ is oriented \emph{out of} $S$.

Let
\begin{equation}
  \label{eq:centmonom}
  \cM_S=\{x_\Sigma \col \Sigma\in \cP^c_S \}\subset \cEh_G
\end{equation}
be the set of monomials in $\K[\Eh]$
corresponding via~\eqref{eq:monom} to partial orientations
in $\cP^c_S$.
The degree of each monomial in $\cM_S$ is equal to $\kap_S=|E_S|$
and there are $|\cM_S|=2^\mu$ of them, where $\mu$ is  the number of non-loop edges with both ends in  $S$.

Consider the ideal
$$
I_G^c:=\Bigl(\bigcup_{S\subset V}\cM_S \Bigr)\subset \cEh_G
$$
generated by the monomials from all $\cM_S $ and
let $\cEh_G^c$ be the quotient algebra
\begin{equation}
  \label{eq:c-partial}
\cEh_G^c:=\cEh_G/I_G^c.
\end{equation}

\begin{theorem}\label{th:monomc}
 The central bizonotopal algebra $\BZ_G^c$ of a graph $G=(V,E)$ is
 isomorphic to the subalgebra of $\cEh_G^c$ generated by  the  degree one elements
 \begin{equation}
   \label{eq:gen-c}
y_v=\sum_{e\in s^{-1}(v)} x_e, \ v\in V,
 \end{equation}
 where $x_e$ is the image in $\cEh_G^c$ of the monomial $x_{\{e\}}\in \cEh_G$.
\end{theorem}

\begin{proof}
 Arguing as in the proof of part (i) of Lemma~\ref{lem:score}, we see that the subalgebra of $\cEh_G^c$ generated by $y_v, v\in V$ is monomial. We need to check that  $\ds \prod_{v\in V} y_v^{a_v}=0$ in $\cEh_G^c$ if and only if  $\ds \prod_{v\in V} z_v^{a_v}=0$ in $\BZ_G^c$.

\medskip
Clearly for any $S\subset V$
and for all
$\ba=(a_v)_{v\in S}\in \Zp^{S}$, with $\ds \sum_{v\in S}a_v=\kap_S$,
we have
$$
\prod_{v\in S} y_v^{a_v}=0\ \text{in}\ \cEh_G^c
$$
which is exactly the set of defining relations for the algebra $\BZ_G^c$.

\medskip
It remains to prove the converse. If $\ds \prod_{v\in V} y_v^{a_v}=0$ in $\cEh_G$, then by Theorem~\ref{th:monom}, we know that the corresponding relation holds for $z_v, \ v\in V$.
Further, assume  that $\ds \prod_{v\in V} y_v^{a_v}=0$ in $\cEh_G^c$, but not in $\cEh_G$. Then the expansion of $\ds \prod_{v\in V} y_v^{a_v}$  in variables   $x_e, e\in \Eh$,
contains a  monomial $m$ that vanishes in $\cEh_G^c$, but not in $\cEh_G$. Therefore $m$ is divisible by a monomial $m'\in \cM_S$ for some $S$. This implies that $\ds \sum_{i\in S} a_i\ge \kap_S$ and, hence, $\ds \prod_{v\in V} z_v^{a_v}=0$ in $\BZ_G^c$.
\end{proof}

\medskip

Recall~\cite{PS,HoRo}  that the dimension of the usual central zonotopal algebra of a connected graph $G$ is equal to the number
of spanning trees of $G$.
The following result gives a similar property of the central bizonotopal algebra.

\begin{theorem}\label{th:Hilbc}
  For a connected graph $G$, the Hilbert series
  \[
    h_G^c(t):=\sum_{k\ge 0}\dim (\BZ_G^c)^{(k)}\, t^k
  \]
  is a polynomial of degree $|E|-1$. Its top coefficient
  $\dim (\BZ_G^c)^{(|E|-1)}$ is equal to the number of spanning trees of $G$.
\end{theorem}

\begin{proof}
  The first statement is immediate from the definition of $\BZ_G^c$.
  For the second, note that $\dim (\BZ_G^c)^{(|E|-1)}$ satisfies the loopy deletion--contraction relation and multiplicativity by Theorem~\ref{delcont:ext-cen}, and agrees with the spanning tree count for one-vertex loop graphs.
  Since the number of spanning trees satisfies the same relations and normalization, the result follows.
\end{proof}

\subsection{Internal bizonotopal algebras}\label{sec:in}

The internal bizonotopal algebra~\eqref{eq:internal}
$ \BZ_G^i:=\BZ_G^{(-1)}$
is defined for all graphs $G=(V,E)$ without isolated vertices.
As in the central case considered in the previous subsection,  $\BZ_G^i$ can be realized as a subalgebra of a certain quotient algebra $\cEh^i_G$ of the partial orientation algebra $\cEh_G$.

Namely, for a subset $S\subset V$, let $ \cM_S^-$ be the set of monomials  $m\in \cEh_G$ such that $x_em\in \cM_S$ for some $e\in \Eh$, where $\cM_S\subset \cEh_G$  is the set of monomials~\eqref{eq:centmonom} introduced in the previous subsection.
Clearly the degree of each  monomial in $\cM_S^-$ is equal to $\kap_S-1$.

Define
$$\cEh_G^i:=\cEh_G/I_G^i, $$
where
$$
I_G^i:=\Bigl( \bigcup_{S\subset V} \cM^-_S \Bigr).$$

\begin{theorem}\label{th:monomi}
 The internal bizonotopal algebra $\BZ_G^i$ of a graph $G=(V,E)$ without isolated vertices is isomorphic to the subalgebra of $\cEh_G^i$ generated by  the linear forms~\eqref{eq:gen-c}
  $$y_v=\sum_{e\in s^{-1}(v)} x_e,$$
  where $x_e$ is the image in $\cEh_G^i$ of the monomial $x_{\{e\}}\in \cEh_G$.
\end{theorem}

\begin{proof} The proof is analogous to that of Theorem~\ref{th:monomc}, with $\cM_S$ replaced by $\cM_S^-$ and all the degrees shifted by $-1$.
\end{proof}

\medskip

The following result (see also Remark~\ref{rem:top-internal} below) provides a nice combinatorial interpretation of the dimension of the
top component of the algebra $\BZ_G^i$ for the complete graph $G=K_n$
similar to those for the external and central bizonotopal algebras
in Corollary~\ref{cor:hilb-e}(4) and Theorem~\ref{th:Hilbc}.

\begin{theorem}\label{th:complete}
For the complete graph $K_n$ with $n\ge 4$, the top degree of $\BZ_{K_n}^i$ is equal to $\binom{n}{2}-2$, and the dimension of the corresponding graded component is $\binom{n-2}{2}n^{n-4}$.
\end{theorem}

\begin{proof}
We linearly order the vertices of $K_n$, i.e.\ identify $V$ with $[n]=\{1,2,\ldots,n\}$.

Since $\ds |E|=\binom{n}{2}$,  all  monomials of degree
$\ds \binom{n}{2}-1=\kap_V-1$ vanish in $\BZ_{K_n}^i$.
Hence the statement reduces to showing that the graded component of degree
$\ds \binom{n}{2}-2$ of the algebra $\BZ_{K_n}^i$ has dimension
$\ds \binom{n-2}{2}n^{n-4}$.

We will derive this from Lemma~\ref{lm:count} below after introducing some auxiliary  combinatorial objects.
\end{proof}

	\medskip
Let us introduce two sets $X$ and $Y$ of integer vectors that parametrize top degree monomials in the internal and central cases, respectively.

The set $X$ consists of all vectors $(b_1,b_2,\ldots,b_n)\in \Zp^n$ satisfying
\begin{itemize}
\item $\ds \sum_{i\in [n]}b_i=\binom{n}{2}-2$;
\item $\ds \sum_{i\in I} b_i\leq \kap_I-2=\binom{|I|}{2}+|I|(n-|I|)-2$, for all $\emptyset\ne I\subset [n]$.
\end{itemize}

The set $Y$ is defined similarly, with the constraints increased by one, that is
$Y$ consists of all  $(b_1,b_2,\ldots,b_n)\in \Zp^n$
such that
\begin{itemize}
\item $\ds \sum_{i\in [n]}b_i=\binom{n}{2}-1$;
\item $\ds \sum_{i\in I} b_i\leq \kap_I-1=\binom{|I|}{2}+|I|(n-|I|)-1$, for all $\emptyset\ne I\subset [n]$.
 \end{itemize}

 By the definition of the internal bizonotopal algebra,
 the dimension of its component of degree $\ds \binom{n}{2}-2$ is equal to $|X|$ which we find below.

 Similarly, $|Y|$ is equal to the dimension of the highest degree component of the central bizonotopal algebra $\BZ^c_{K_n}$. Thus, by  Theorem~\ref{th:Hilbc},  $|Y|$ is equal to the number of spanning trees in $K_n$, that is $|Y|=n^{n-2}$.

 From the inequalities defining $X$ and $Y$ we see that if $(b_1,\ldots,b_{n-1},b_n)\in X$, then $(b_1,\ldots,b_{n-1},b_n+1)\in Y$.
Therefore we have
 \begin{equation*}
   \begin{split}
 |X| &=|Y|-|\{(b_1,b_2,\ldots,b_n)\in Y:\ b_n=0\}|\\
 & -|\{(b_1,b_2,\ldots,b_n)\in Y:\ b_n>0\ \text{and}\ \sum_{j\in J}b_j=\kap_J-1\ \text{for some}\ J\subset[n-1]\}|\\
 &=|Y|-|\{(b_1,b_2,\ldots,b_n)\in Y:\ \sum_{j\in J}b_j=\kap_J-1\ \text{for some}\ J\subset[n-1]\}|.
   \end{split}
 \end{equation*}

 Denote by $Z$ the subset of $Y$ consisting of vectors that saturate at least one inequality indexed by a subset of $[n-1]$:
 $$
 Z:=\{(b_1,b_2,\ldots,b_n)\in Y \col \sum_{j\in J}b_j=\kap_J-1\  \text{for some}\ J\subset[n-1]\}.
 $$
We claim that for every $(b_1,b_2,\ldots,b_n)\in Z$, there exists a unique maximal subset $J\subset [n-1]$  such that $\ds \sum_{j\in J} b_j=\kap_{J}-1$.
 Indeed assume that  there exist two different sets
 $J_1, J_2\subset[n-1]$
 satisfying
 $$ \sum_{j\in J_1} b_j=\kap_{J_1}-1 \text{ and } \sum_{j\in J_2} b_j=\kap_{J_2}-1$$
 such that $J_1\not\subset J_2$ and $J_2\not\subset J_1$.
If $J_1\cap J_2\ne \emptyset$, then we have
\begin{equation*}
 \begin{split}
   \sum_{j\in J_1\cup J_2} b_j & =
    \sum_{j\in J_1} b_j  +\sum_{j\in J_2} b_j-\sum_{j\in J_1\cap J_2} b_j\\
     & = \kap_{J_1}-1 + \kap_{J_2}-1 - \sum_{j\in J_1\cap J_2} b_j\\
     & \ge \kap_{J_1}-1 + \kap_{J_2}-1-(\kap_{J_1\cap J_2}-1)
                                \ge \kap_{J_1\cup J_2}-1,
 \end{split}
\end{equation*}
where the last inequality follows from the submodularity of the function $\kap$ (see Lemma~\ref{lem:submodular}).
If $J_1\cap J_2=\emptyset$, then
$$ \sum_{j\in J_1\cup J_2} b_j=\sum_{j\in J_1} b_j+\sum_{j\in J_2} b_j
= \kap_{J_1}-1 +\kap_{J_2}-1\ge \kap_{J_1\cup J_2}-1.
$$
Therefore, $\ds \sum_{j\in J_1\cup J_2} b_j \ge \kap_{J_1\cup J_2}-1$. Since, by definition of the set $Y$, we also have $\ds \sum_{j\in J_1\cup J_2} b_j \le \kap_{J_1\cup J_2}-1$, we see that the union $J_1\cup J_2$
also satisfies our property, i.e.\ neither $J_1$ nor $J_2$ are maximal  which contradicts our assumption.

\medskip

Therefore we have a partition
$\ds Z=\bigsqcup_{\emptyset\ne I\subset [n-1]} Z_I$
into the subsets
$$
Z_I:=\{(b_1,\ldots,b_n)\in Z: \sum_{i\in I}b_i=\kap_I-1
\text{ and } \sum_{j\in J}b_j<\kap_J-1
\text{ if } I\subsetneq J \subset[n-1]\}.$$
Thus
$$|X|=|Y|-|Z|=n^{n-2} - \sum_{\emptyset\ne I\subset [n-1]} |Z_I|,$$
and the problem reduces to computing the cardinalities of the subsets $Z_I$.

\begin{lemma}\label{lm:count}
For any subset $J\subset [n-1]$, we have $$|Z_J|=|J|^{|J|-2}\bigl(n-|J|\bigr)^{(n-|J|)-2},$$
which coincides with the product of the numbers of spanning trees of the complete graphs on $J$ and on its complement $[n]\setminus J$.
\end{lemma}

\begin{proof}
Let us list the elements of the subsets $J\subset [n]$ and
$[n]\setminus J$ in order, i.e.\
$J=\{j_1,j_2,\ldots,j_{\ell}\}$ and $[n] \setminus  J =\{i_1,i_2,\ldots,i_{n-\ell}\}$ with   $j_1<j_2<\ldots <j_{\ell}$ and
  $i_1 < i_2 \ldots < i_{n-\ell}$.
\medskip
Define two  maps  $\phi_1:\Z^n \to \Z^{\ell}$ and $ \phi_2:\Z^n \to \Z^{n-\ell}$ by
$$
\phi_1(b):=(b_{j_1}-n+\ell,b_{j_2}-n+\ell,\ldots, b_{j_{\ell}}-n+\ell),
$$
and
$$
\phi_2(b):=(b_{i_1},b_{i_2}, \ldots,b_{i_{n-\ell-1}},b_{i_{n-\ell}}-1).
$$

To prove the lemma, we will show that $b\in Z_J$ if and only if
$\phi_1(b)\in Y_{\ell}$ and $\phi_2(b)\in Y_{n-\ell}$, where by $Y_{\ell}$ and $Y_{n-\ell}$ we denote the sets $Y$ for complete graphs $K_{\ell}$ and $K_{n-\ell}$ respectively.

Assume that $b\in Z_J$. Then for any $I\subset [\ell]$, we have
\begin{equation*}
  \begin{split}
\sum_{t\in I} (\phi_1(b))_t& = \sum_{t\in I} (b_{j_t}-n+\ell)
                    =\bigl(\sum_{t\in I} b_{j_t}\bigr)+|I|(-n+\ell)
 \leq
    \kap_I-1+|I|(-n+\ell)\\
  &=\binom{|I|}{2} +|I|(n-|I|)-1+|I|(-n+\ell)
  =\binom{|I|}{2} +|I|(\ell-|I|)-1.
\end{split}
\end{equation*}
Hence $\phi_1(b)\in Y_{\ell}$.

\medskip
For any subset $I\subset [n-\ell-1]$, we have
\begin{equation*}
  \begin{split}
\sum_{t\in I} \phi_2(b)_t & = \sum_{j\in \{i_t: t\in I\}} b_{j}
    =\sum_{j\in \{i_t,t\in I\} \sqcup J} {b_j}- \sum_{j\in J} {b_j}\\
                          & = \sum_{j\in \{i_t,t\in I\} \sqcup J} {b_j}-(\kap_J-1)< (\kap_{\{i_t,t\in I\} \sqcup J}-1)-(\kap_J-1)\\
    &=\binom{|I|}{2}+|I|(n-\ell-|I|)
  \end{split}
\end{equation*}
and for all $I\subset [n-\ell]$ s.t. $(n-\ell)\in I$, we have
\begin{equation*}
  \begin{split}
\sum_{t\in I} \phi_2(b)_t & = \sum_{j\in \{i_t: t\in I\}} b_{j}-1
    =\sum_{j\in \{i_t,t\in I\} \sqcup J} {b_j}- \sum_{j\in J} {b_j}-1\\
                          & = \sum_{j\in \{i_t,t\in I\} \sqcup J} {b_j}-\kap_J\leq \kap_{\{i_t,t\in I\} \sqcup J}-\kap_J-1\\
    & =\binom{|I|}{2}+|I|(n-\ell-|I|)-1.
  \end{split}
\end{equation*}
Hence, $\phi_2(b)\in Y_{n-\ell}$.

\medskip
Let us  prove the converse. Assume that $\phi_1(b)\in Y_{\ell}$ and $\phi_2(b)\in Y_{n-\ell}$. We need to show that for any $I\subset [n]$, one has $\ds \sum_{i\in I} b_i\leq \binom{|I|}{2}+ |I|(n-|I|)-1$.
To do this we will consider three cases.
\medskip

\noindent {\it Case 1: $I \cap J=\emptyset$}.
\
Since $\phi_2(b)\in Y_{n-\ell}$, we have
\begin{equation*}
  \begin{split}
    \sum_{i\in I} b_i & \leq \binom{|I|}{2} +|I|(n-\ell-|I|)-1+1\\
    & = \binom{|I|}{2} +|I|(n-|I|)-n\ell\leq \binom{|I|}{2} +|I|(n-|I|)-1.
  \end{split}
\end{equation*}

\noindent {\it Case 2: $I \cap ([n] -  J)=\emptyset$}.
\
Since $\phi_1(b)\in Y_{\ell}$, we have
\begin{equation*}
  \begin{split}
    \sum_{i\in I} b_i & =  \sum_{i\in I} (b_i-n+\ell) +|I|(n-\ell) \\
                      & \leq \binom{|I|}{2} +|I|(\ell-|I|)-1+ |I|(n-\ell)=\binom{|I|}{2} +|I|(n-|I|)-1.
  \end{split}
\end{equation*}

\noindent {\it Case 3: $I \cap J\neq \emptyset$,  $I \cap ([n] -  J)\ne \emptyset$}.
\
We have
\begin{equation*}
  \begin{split}
    \sum_{i\in I} b_i & = \sum_{i\in I\cap J} b_i + \sum_{i\in I\cap ([n] - J)} b_i  \\
                      & \leq \binom{|I\cap J|}{2}  +|I\cap J|(n-|I\cap J|)-1
                        +\binom{|I\cap ([n] -  J)|}{2} \\
                      & \phantom{\leq \binom{|I\cap J|}{2} \ } + |I\cap ([n] -  J)|(n-\ell-|I\cap ([n] -  J)|)-1+1\\
                      &=\binom{|I|}{2} +|I|(n-|I|)-1.
  \end{split}
\end{equation*}
We also need to check that for $I=J$, we have the non-strict inequality, i.e.\
\[ \sum_{i\in J} b_i\leq \binom{|J|}{2}+ |J|(n-|J|)-1.
\]
It follows from the second case.

It remains to check that for all $I$ with $J\subsetneq I\subset [n-1]$, we have the strict inequality. Indeed this happens in the third case and we do not use $b_n$, so we have the same inequality, but without ``$+1$''.

Thus we showed that $b\in Z_J$ if and only if $\ds \phi_1(b)\in Y_{\ell}$ and $\phi_2(b)\in Y_{n-\ell}$.
It can be verified directly that the map
$(\phi_1, \phi_2): \Z^{n}\to \Z^{\ell}\times \Z^{n-\ell}$ is a bijection.
This implies that $|Z_J|=|Y_{\ell}|\times |Y_{n-\ell}|$.
Since the numbers of elements in $Y_{\ell}$ and $Y_{n-\ell}$
are equal to the numbers of spanning trees in
$K_{\ell}$ and $K_{n-\ell}$, respectively, we obtain
$|Z_J|=|J|^{|J|-2}(n-|J|)^{(n-|J|)-2}$
as required.
\end{proof}

\begin{proof}[End of the proof of Theorem~\ref{th:complete}]
  By the above lemma, $|Z_J|$ is equal to the product of the numbers of spanning trees in $K_J$ and in $K_{[n]- J}$. This implies that
  $$|Z| = \sum_{\emptyset \neq I\subset [n-1]} |Z_I|$$
counts the number of spanning forests in $K_n$ with two connected components (i.e.\ forests with exactly $n-2$ edges).
  As was proved by R\'{e}nyi~\cite{Re} (see also~\cite[Eq.~(19)]{LC} or \cite[Theorem 2.2]{My}),   the number of two-component spanning forests in $K_n$ is equal to $n^{n-4}\cdot\frac{(n-1)(n+6)}{2}$.
  Thus we obtain that
  $$
  |X|=n^{n-2}-n^{n-4}\cdot\frac{(n-1)(n+6)}{2}
  = n^{n-4} \cdot \frac{(n-2)(n-3)}{2} =n^{n-4}\binom{n-2}{2}.
  $$
\end{proof}

\begin{remark} \label{rem:top-internal}
  From the proof of Theorem~\ref{th:complete} we see that
  the dimension of the top degree component of the algebra $\BZ_{K_n}^i$ has a combinatorial interpretation as the difference between the number of spanning trees in $K_n$ and the number of two-component spanning forests in $K_n$.
  It would be very interesting to find a combinatorial description of the dimension of the top degree component of $\BZ_G^i$ for general graphs.
  (Notice that in general, the difference between the numbers of spanning trees and two-component spanning forests of a graph may not even be positive.)
\end{remark}

\medskip

Unlike the external and the central bizonotopal algebras, the internal algebras provide much weaker graph invariants. In fact, as the following theorem shows, their Hilbert series coincide for large classes of graphs.

\begin{theorem}\label{th:3-4}
  Let $G=(V,E)$ be a simple graph with $n\ge 4$ vertices and let $h^i_G(t)$ be the Hilbert series of its internal bizonotopal algebra $\BZ^i_G$.
\\[4pt]
(i) If $G$ is a $3$-regular graph (in which case $n\ge 4$ and is even) then
$$ h_G^i=(1+t)^n.$$
(ii) If $G$ is  a $4$-regular $4$-edge-connected graph (so that $n\ge 5$), then
$$h_G^i=(1+t+t^2)^{n}-nt^{2n-1}-t^{2n}.$$
\end{theorem}
\begin{proof}
  (i) If $G$ is a $3$-regular graph, then for each $v\in V$,  we have $\kap_v=3$ and thus $z_v^2=0$ in $\BZ_G^i$.
  Moreover, the  monomial $\ds \prod_{v\in V}z_v$ of degree $n$ does not vanish in $\BZ_G^i$, because for any $I\subset V$ we have
  \[ \kap_I-1\ge \frac{3}{2}|I|-1 >|I|.
  \]
  Hence, the ideal of relations of $\BZ_G^i$ does not contain any square-free monomials and therefore $h_G^i(t)=(1+t)^n$.

  \medskip

\noindent
(ii) If $G$ is $4$-regular, then  $\kap_v=4$ for all $v\in V$ and we have $z_v^3=0$ in $\BZ_G^i$.
We also have $\ds \prod_{v\in V}z_v^2=0$ in $\BZ_G^i$, because $\kap_V=|E|=2n$.
For any $u\in V$, the product $\ds \prod_{v\ne u}z_v^2$ does not vanish in $\BZ_G^i$, because for a proper  subset $I\subsetneq V$, by $4$-edge-connectedness, we have
\[
\kap_I-1\ge \frac{4}{2}(|I|-4)+4-1=2|I|+1.
\]
Therefore
\[ h_G^i(t)=(1+t+t^2)^{n}-nt^{2n-1}-t^{2n},
\]
where the subtracted terms correspond to the vanishing of monomials of total degree $2n-1$ and $2n$.
\end{proof}

\section{Concluding remarks}\label{sec:conclude}
Below we list several questions about bizonotopal algebras that, in our opinion, warrant further investigation.

\begin{enumerate}[(1)]
\item The dimension of the external bizonotopal algebra $\BZ_G^e$ of a graph $G$ has a nice combinatorial interpretation as the number of integer points in the polytope of weak parking functions of $G$. Finding a combinatorial  interpretation for the dimension of the central algebra $\BZ^c_G$ is an interesting problem. (Currently we only know such an interpretation for the dimension of the top degree component of $\BZ^c_G$.)

\item  It would be interesting to find a combinatorial interpretation of the dimension of the top-degree component of the internal bizonotopal algebra $\BZ_G^i$ for graphs other than $K_n$.

\item The Hilbert series of the internal bizonotopal algebras, unlike the central and external cases, do not satisfy the loopy deletion-contraction relation. Does it satisfy a recursion of some other kind?

\item The external and central bizonotopal algebras are very strong (almost complete) graph invariants. On the other hand, we saw that the internal algebra is a rather weak invariant. It would be interesting to characterize pairs of graphs with isomorphic $\BZ_G^i$.

\item
In a follow-up paper~\cite{KNSV} we show that using our loopy deletion-contraction relation~\eqref{eq:del-contr} one can construct a new  multivariate graph polynomial which we call the \emph{loopy polynomial}.
The Hilbert series of $r$-bizonotopal algebras are certain specializations of this loopy polynomial.
It strongly resembles Stanley's Tutte symmetric function and several related polynomials, but the precise relationship between them is currently unclear.  This aspect needs clarification.

\item
Even though the Hilbert series of the bizonotopal algebras are not specializations of the Tutte polynomial, our computer calculations indicate that they are all unimodal and log-concave.
Proving these properties and determining whether they hold for other specializations of the loopy polynomial seems to be a very interesting problem.
 \end{enumerate}

 \section{Appendix: Hilbert functions of bizonotopal algebras of complete graphs}
 \label{sec:num}
    Below we present the results of the computations for complete graphs $K_n$ with $n\le 9$ vertices of the dimensions of the bizonotopal algebras $\BZ^e_{K_n}$, $\BZ^c_{K_n}$ and $\BZ^i_{K_n}$  and their Hilbert function  $h(k)$, the dimension of the $k$th graded  component $\BZ^{(k)}$ of the corresponding algebra.

\subsection{External algebras $\BZ_{K_n}^e$}

  \noindent $K_2$:\ \ $\dim = 3$; \quad $h(k)$: \  {1, 2};

  \  \\ \noindent $K_3$:\ \ $\dim = 17$; \quad $h(k)$: \  {1, 3, 6, 7};

  \  \\ \noindent $K_4$:\ \ $\dim = 144$; \quad $h(k)$: \ {1, 4, 10, 20, 31, 40, 38};

  \  \\ \noindent $K_5$:\ \ $\dim = 1623$; \quad $h(k)$: \ {1, 5, 15, 35, 70, 121, 185, 255, 310, 335, 291};

  \  \\ \noindent $K_6$:\ \ $\dim = 22804$; \quad $h(k)$: \ {1, 6, 21, 56, 126, 252, 456, 756, 1161, 1666, 2232, 2796, 3281, 3546, 3516, 2932};

  \  \\ \noindent $K_7$:\ \ $\dim = 383415$; \quad $h(k)$: \ {1, 7, 28, 84, 210, 462, 924, 1709, 2954, 4809, 7420, 10906, 15309, 20559, 26454, 32655, 38591,
43589, 46984, 47649, 45150, 36961};

  \  \\ \noindent $K_8$:\ \ $\dim =  7501422$; \quad $h(k)$: \ {1, 8, 36, 120, 330, 792, 1716, 3432, 6427, 11376, 19160, 30864, 47748, 71184, 102524, 142920,
 193117, 253240, 322596, 399344, 480390, 561472, 637400, 701296, 746089, 765640, 748532,
691720, 561948};

  \  \\ \noindent $K_9$:\ \ $\dim = 167341283$; \quad $h(k)$: \ {1, 9, 45, 165, 495, 1287, 3003, 6435, 12870, 24301, 43677, 75177, 124485, 199035, 308187, 463287, 677520, 965493, 1342513, 1823553, 2421927, 3147723, 4005819, 4993839, 6100350, 7303545,
8570601, 9855829, 11101599, 12241305, 13203705, 13902291, 14254524, 14195199, 13575951,  12369033, 10026505}.

\subsection{Central algebras $\BZ_{K_n}^c$}

 \noindent $K_2$:\ \ $\dim = 1$; \quad $h(k)$: \ {1};

  \  \\ \noindent $K_3$:\ \ $\dim = 7$; \quad $h(k)$: \ {1, 3, 3};

  \  \\ \noindent $K_4$:\ \ $\dim = 66$; \quad $h(k)$: \ {1, 4, 10, 16, 19, 16};

  \  \\ \noindent $K_5$:\ \ $\dim = 792$; \quad $h(k)$: \ {1, 5, 15, 35, 65, 101, 135, 155, 155, 125};

  \  \\ \noindent $K_6$:\ \ $\dim = 11590$; \quad $h(k)$: \ {1, 6, 21, 56, 126, 246, 426, 666, 951, 1246, 1506, 1686, 1731, 1626, 1296};

  \  \\ \noindent $K_7$:\ \ $\dim = 200469$; \quad $h(k)$: \ {1, 7, 28, 84, 210, 462, 917, 1667, 2807, 4417, 6538, 9142, 12117, 15267, 18327, 20958, 22827, 23667, 23107, 21112, 16807};

  \  \\ \noindent $K_8$:\ \ $\dim = 90759016$; \quad $h(k)$: \ {1, 8, 36, 120, 330, 792, 1716, 3424, 6371, 11152, 18488, 29184, 44052, 63792, 88852, 119288, 154645, 193880, 235292, 276592, 315078, 347880, 371820, 384112, 382817, 364232, 328392, 262144};

  \  \\ \noindent $K_9$:\ \ $\dim = 2301604074$; \quad $h(k)$: \ {1, 9, 45, 165, 495, 1287, 3003, 6435, 12861, 24229, 43353, 74097, 121515, 191907, 292743, 432399, 619677, 863109, 1170073, 1545777, 1992195, 2506983, 3082599, 3705795, 4357593, 5013801, 5645313, 6219649, 6703245, 7064073, 7267815, 7285959, 7100739, 6660495, 5966613, 4782969}.

\subsection{Internal algebras $\BZ_{K_n}^i$}

 \noindent $K_3$:\ \ $\dim = 1$; \quad $h(k)$: \ {1};

  \  \\ \noindent $K_4$:\ \ $\dim = 16$; \quad $h(k)$: \ {1, 4, 6, 4, 1};

  \  \\ \noindent $K_5$:\ \ $\dim = 237$; \quad $h(k)$: \ {1, 5, 15, 30, 45, 51, 45, 30, 15};

  \  \\ \noindent $K_6$:\ \ $\dim = 3892$; \quad $h(k)$: \ {1, 6, 21, 56, 120, 216, 336, 456, 546, 580, 546, 456, 336, 216};

  \  \\ \noindent $K_7$:\ \ $\dim = 72425$; \quad $h(k)$: \ {1, 7, 28, 84, 210, 455, 875, 1520, 2415, 3535, 4795, 6055, 7140, 7875, 8135, 7875, 7140, 6055, 4795, 3430};

  \  \\ \noindent $K_8$:\ \ $\dim = 1521810$; \quad $h(k)$: \ {1, 8, 36, 120, 330, 792, 1708, 3368, 6147, 10480, 16808, 25488, 36688, 50288, 65808, 82384, 98813, 113688, 125588, 133288, 135954, 133288, 125588, 113688, 98533, 81488, 61440};

  \  \\ \noindent $K_9$:\ $\dim = 35794801$; \ $h(k)$: \ {1, 9, 45, 165, 495, 1287, 3003, 6426, 12789, 23905, 42273, 71127, 114387, 176463, 261891, 374808, 518301, 693693, 899857, 1132677, 1384803, 1645791, 1902663, 2140866, 2345553, 2503053, 2602341, 2636263, 2602341, 2502423, 2342907, 2134062, 1881243, 1596861, 1240029}.

\end{document}